\documentclass{article}
\usepackage{amsmath,amssymb,amstext,amsthm}
\usepackage[utf8]{inputenc}
\usepackage[T1]{fontenc}
\usepackage{graphicx}
\usepackage[mathcal]{euscript}

\newcommand{\m}{\odot} 

\newcommand{\Av}{\text{Av}} 
\newcommand{\U}[1]{{\mathcal U}_{#1}} 
\newcommand{\ma}[1]{\widehat{#1}} 
\newcommand{\MAv}{\text{MAv}} 
\newcommand{\cro}{\text{cr}} 
\newcommand{\wei}{\text{w}} 
\newcommand{\lo}{\ell}

\newcommand{\lend}[1]{\text{left}(#1)}
\newcommand{\rend}[1]{\text{right}(#1)}

\newcommand{\cA}{A}
\newcommand{\cB}{B}
\newcommand{\cC}{C}
\newcommand{\bbN}{{\mathbb N}}
\newcommand{\bbR}{{\mathbb R}}

\newcommand{\cS}{{\mathcal S}} 

\newcommand{\dint}[1]{\overrightarrow{G_{#1}}}
\newcommand{\uint}[1]{G_{#1}}

\newcommand{\J}{\delta} 

\newcommand{\mexp}[2]{{#2}^{*#1}}

\newtheorem{lemma}{Lemma}[section]
\newtheorem{theorem}[lemma]{Theorem}
\newtheorem{observation}[lemma]{Observation}                              
\newtheorem{proposition}[lemma]{Proposition}
\newtheorem{claim}{Claim}
\newtheorem{corollary}[lemma]{Corollary}
\newtheorem{fact}[lemma]{Fact}

\theoremstyle{definition}
\newtheorem{definition}[lemma]{Definition}
\newtheorem{problem}{Problem}

\title{Splittings and Ramsey Properties of Permutation Classes\footnote{
This paper is intended exclusively for arXiv and will not be published
elsewhere. Parts of it, however, may be published separately.}
}
\author{Vít Jelínek\thanks{
Computer Science Institute, Charles University in Prague,
\texttt{jelinek@iuuk.mff.cuni.cz}} 
\and Pavel Valtr\thanks{
Department of Applied Mathematics, Faculty of Mathematics and Physics, Charles
University in Prague, \texttt{valtr@kam.mff.cuni.cz}
}
}

\begin{document}
\maketitle
\begin{abstract}
We say that a permutation $\pi$ is \emph{merged} from permutations $\rho$ and
$\tau$, if we can color the elements of $\pi$ red and blue so that
the red elements are order-isomorphic to $\rho$ and the blue ones to~$\tau$. 
A \emph{permutation class} is a set of permutations closed under taking
subpermutations. A permutation class $\cC$ is \emph{splittable} if it has two
proper subclasses $\cA$ and $\cB$ such that every element of $\cC$ can be
obtained by merging an element of $\cA$ with an element of~$\cB$.

Several recent papers use splittability as a tool in deriving enumerative
results for specific permutation classes. The goal of this paper is to study
splittability systematically. As our main results, we show that if $\sigma$ is
a sum-decomposable permutation of order at least four, then the class
$\Av(\sigma)$ of all $\sigma$-avoiding permutations is splittable, while if
$\sigma$ is a simple permutation, then $\Av(\sigma)$ is unsplittable.

We also show that there is a close connection between splittings of certain
permutation classes and colorings of circle graphs of bounded clique size.
Indeed, our splittability results can be interpreted as a generalization of a
theorem of Gyárfás stating that circle graphs of bounded clique size have
bounded chromatic number.
\end{abstract}

\section{Introduction}

The study of pattern-avoiding permutations, and more generally, of hereditary
permutation classes, is one of the main topics in combinatorics. However,
despite considerable effort, many basic questions remain unanswered. For
instance, for permutations that avoid the pattern 1324, we still have no
useful structural characterization, and no precise asymptotic enumeration
either. 

Recently, Claesson, Jelínek and Steingrímsson~\cite{cjs} have shown that every
permutation $\pi$ that avoids 1324 can be merged from a permutation avoiding
132 and a permutation avoiding 213; in other words, the elements of $\pi$ can
be colored red and blue so that there is no red copy of 132 and no blue copy of
213. From this, they deduced that there are at most $16^n$ 1324-avoiding
permutations of order~$n$. They have extended this merging argument to more
general patterns, showing in particular that if $\sigma$ is a layered pattern
of size $k$, then there are at most $(2k)^{2n}$ $\sigma$-avoiding permutations
of size~$n$. Subsequently, this approach was further developed by
B\'ona~\cite{bona-1324,bona-newlay}, who proved, among other results, that there
are at most $(7+4\sqrt{3})^n\simeq 13.93^n$ 1324-avoiders of size~$n$. These
results are also based on arguments showing that avoiders of certain patterns
can be merged from avoiders of smaller patterns.

Motivated by these results, we address the general problem of identifying when
a permutation class $\cC$ has proper subclasses $\cA$ and $\cB$, such that every
element of $\cC$ can be obtained by merging an element of $\cA$ with an element
of~$\cB$. We call a class $\cC$ with this property \emph{splittable}. In this
paper, we mostly focus on classes defined by avoidance of a single forbidden
pattern, although some of our results are applicable to general hereditary
classes as well.

On the negative side, we show that if $\sigma$ is a simple permutation, then the
class $\Av(\sigma)$ of all $\sigma$-avoiding permutations is unsplittable. More
generally, every wreath-closed permutation class is unsplittable. We also find
examples of unsplittable classes that are not wreath-closed, e.g., the class of
layered permutations or the class of 132-avoiding permutations.

On the positive side, we show that if $\sigma$ is a direct sum of two nonempty
permutations and has size at least four, then $\Av(\sigma)$ is splittable. This
extends previous results of Claesson et al.~\cite{cjs}, who address the
situation when $\sigma$ is a direct sum of three permutations, with an extra
assumption on one of the three summands.

The concept of splittability is closely related to several other structural 
properties of classes of relational structures, which have been previously
studied in the area of Ramsey theory. We shall briefly mention some of these
connections in Subsection~\ref{ssec-ramsey}. 

We will also establish a less direct, but perhaps more useful, connection 
between splittability and coloring of circle graphs. Let $\sigma_k$ be the 
permutation $1k(k-1)\dotsb32$ of order~$k+1$. We will show, as a special 
case of more general results, that all $\sigma_k$-avoiding permutations can be
merged from a bounded number, say $f(k)$, of $132$-avoiding permutations.
Moreover, we prove that the smallest such $f(k)$ is equal to the smallest number
of colors needed to properly color every circle graph with no clique of 
size~$k$. This allows us to turn previous results on circle
graphs~\cite{gyar,koslow,koskra} into results on splittability of
$\sigma_k$-avoiding permutations, and to subsequently extend these results to
more general patterns. We deal with this topic in Subsection~\ref{ssec-circ}.

\subsection{Basic notions}\label{ssec-basics}

\subsubsection*{Permutation containment}
A \emph{permutation} of order $n\ge 1$ is a sequence $\pi$ of $n$ distinct
numbers from the
set $[n]=\{1,2,\dotsc,n\}$. We let $\pi(i)$ denote the $i$-th element of~$\pi$.
We often represent a permutation $\pi$ by a \emph{permutation diagram},
which is a set of $n$ points with Cartesian coordinates $(i,\pi(i))$, for
$i=1,\dotsc,n$. The set of permutations of order~$n$ is denoted by $S_n$. When
writing out short permutations explicitly, we omit all punctuation and write,
e.g., 1324 for the permutation 1,~3,~2,~4.

The \emph{complement} of a permutation $\pi\in S_n$ is the permutation
$\sigma\in S_n$ satisfying $\sigma(i)=n-\pi(i)+1$. For a permutation $\pi\in
S_n$ and two indices $i,j\in[n]$, we say that the element $\pi(i)$
\emph{covers} $\pi(j)$ if $i<j$ and $\pi(i)<\pi(j)$. If an element $\pi(i)$ is
not covered by any other element of $\pi$, i.e., if $\pi(i)$ is the smallest
element of $\pi(1),\pi(2),\dotsc,\pi(i)$, we say that $\pi(i)$ is a
\emph{left-to-right minimum}, or just \emph{LR-minimum} of~$\pi$. If $\pi(i)$
is not an LR-minimum, we say that $\pi(i)$ is a \emph{covered element} of~$\pi$.

Given two permutations $\sigma\in S_m$ and $\pi\in S_n$, we say that $\pi$
\emph{contains} $\sigma$ if there is an $m$-tuple of indices $1\le
i_1<i_2<\dotsb<i_m\le n$, such that the sequence $\pi(i_1),\dotsc,\pi(i_m)$ is 
\emph{order-isomorphic} to $\sigma$, i.e., if for every $j,k\in[m]$ we have 
$\sigma(j)<\sigma(k)\iff \pi(i_j)<\pi(i_k)$. We then say that $\pi(i_1),\dotsc,
\pi(i_m)$ is an \emph{occurrence} of $\sigma$ in $\pi$, and the function
$f\colon[m]\to[n]$ defined by $f(j)=i_j$ is an \emph{embedding} of $\sigma$ 
into~$\pi$. A permutation that does not contain $\sigma$ is
\emph{$\sigma$-avoiding}. We let $\Av(\sigma)$ denote the set of all
$\sigma$-avoiding permutations, and for a set $F$ of permutations, we let
$\Av(F)$ denote the set of permutations that avoid all elements of~$F$.

A set $\cC$ of permutations is \emph{hereditary} if for every $\pi\in \cC$ all
the permutations contained in $\pi$ belong to $\cC$ as well. We use the term
\emph{permutation class} to refer to a hereditary set of permutations. It is
not hard to see that a set $\cC$ of permutations is hereditary if and only if
there is a (possibly infinite) set $F$ such that $\cC=\Av(F)$.
A \emph{principal} permutation class is a class of the form $\Av(\pi)$ for some
permutation~$\pi$.

\subsubsection*{Direct sums and inflations}
The \emph{direct sum} $\sigma\oplus\pi$ of two permutations $\sigma\in S_m$ and
$\pi\in S_n$ is the permutation $\sigma(1),\sigma(2),\dotsc,\sigma(m),\pi(1)+m,
\pi(2)+m,\dotsc,\pi(n)+m\in S_{n+m}$. Similarly, the \emph{skew sum}
$\sigma\ominus\pi$ is the permutation
$\sigma(1)+n,\sigma(2)+n,\dotsc,\sigma(m)+n, \pi(1),\pi(2),\dotsc,\pi(n)$ (see
Figure~\ref{fig-suma}.
A permutation is \emph{decomposable} if it is a direct sum of two nonempty
permutations. 

\begin{figure}
\hfil\includegraphics[width=0.7\textwidth]{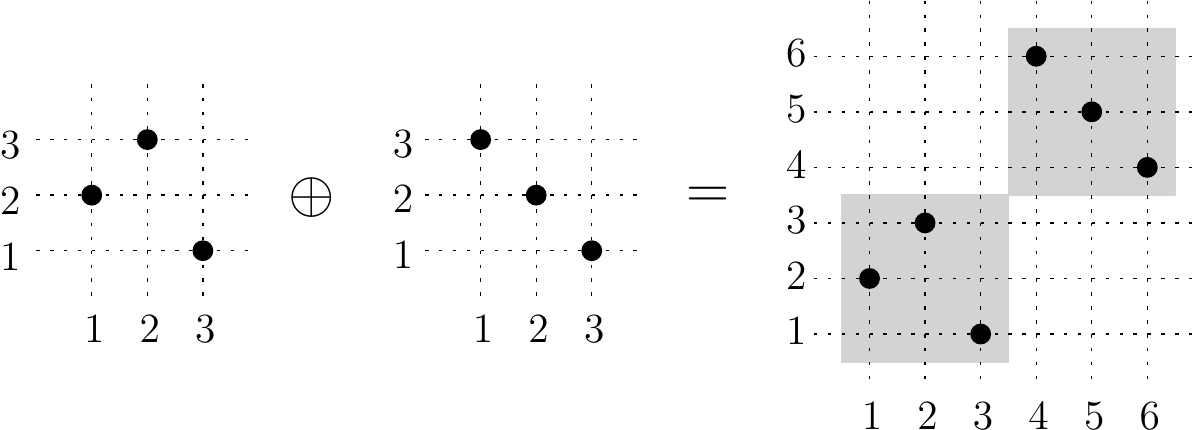}
 \caption{An example of a direct sum: $231\oplus321=231654$}\label{fig-suma}
\end{figure}

Suppose that $\pi\in S_n$ is a permutation, let $\sigma_1,\dotsc,\sigma_n$
be an $n$-tuple of nonempty permutations, and let $m_i$ be the order of
$\sigma_i$ for~$i\in[n]$. \emph{The inflation} of $\pi$ by the sequence
$\sigma_1,\dotsc,\sigma_n$, denoted by $\pi[\sigma_1,\dotsc,\sigma_n]$, is the
permutation of order $m_1+\dotsb+m_n$ obtained by concatenating $n$ sequences
$\overline{\sigma}_1\overline{\sigma}_2\dotsb\overline{\sigma}_n$ with these
properties (see Figure~\ref{fig-naf}):
\begin{itemize}
 \item for each $i\in[n]$, $\overline{\sigma}_i$ is order-isomorphic to
$\sigma_i$, and
\item for each $i,j\in[n]$, if $\pi(i)<\pi(j)$, then all the elements of
$\overline{\sigma}_i$ are smaller than all the elements
of~$\overline{\sigma}_j$.
\end{itemize}

\begin{figure}
\hfil\includegraphics[width=0.7\textwidth]{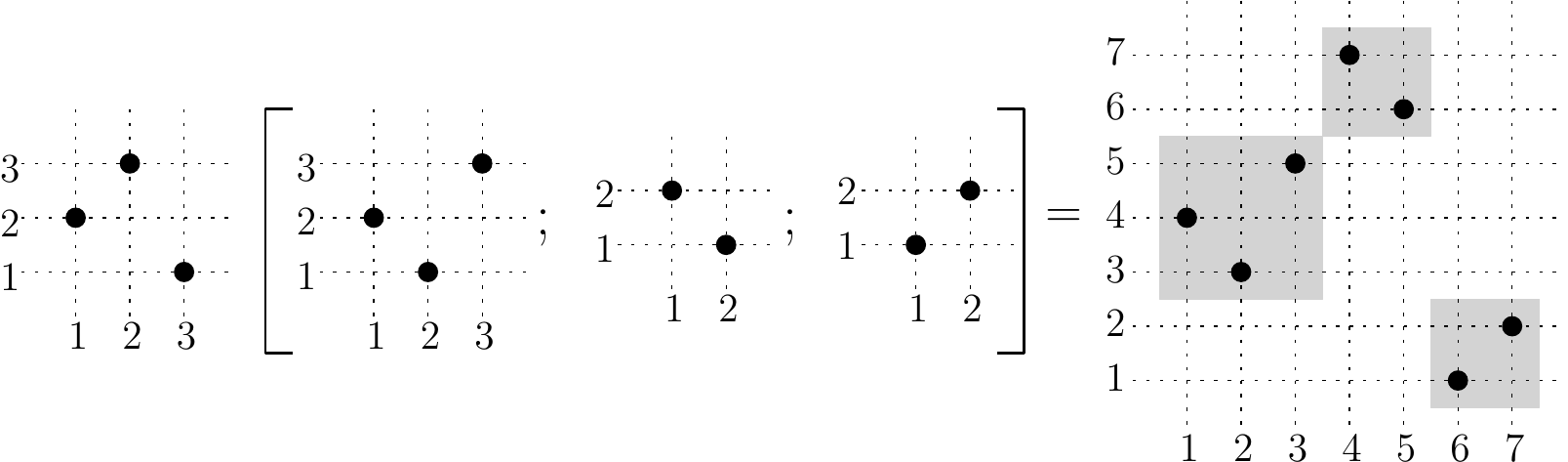}
 \caption{An example of inflation: $231[213,21,12]=4357612$}\label{fig-naf}
\end{figure}

We say that a permutation $\pi$ is \emph{simple} if it cannot be obtained
by an inflation, except for the trivial inflations $\pi[1,1,\dotsc,1]$ and
$1[\pi]$. Albert and Atkinson~\cite[Proposition 2]{alat_sim} have pointed out
that for every permutation $\rho$ there is a unique simple permutation $\pi$
such that $\rho$ may be obtained by inflating $\pi$, and moreover, if $\pi$ 
is neither 12 nor 21, then the inflation is determined uniquely. 

For two sets $\cA$ and $\cB$ of permutations, \emph{the wreath product} of $\cA$
and $\cB$, denoted by $\cA\wr\cB$, is the set of all the permutations that may
be obtained by inflating an element of $\cA$ by a sequence of elements of~$\cB$.
If $\cA$ or $\cB$ is a singleton set $\{\rho\}$, we just write $\rho\wr
\cB$ or $\cA\wr \rho$, respectively. Note that if $\cA$ and $\cB$ are
hereditary, then so is~$\cA\wr\cB$.

For a set $X$ of permutations, we say that a set of permutations $Y$ is 
\emph{closed under $\wr X$} if $Y\wr X \subseteq Y$, and $Y$ is \emph{closed 
under $X\wr$} if $X\wr Y\subseteq Y$. Note that $Y$ is closed under $12\wr$
if and only if it is closed under taking direct sums, and it is closed under
$21\wr$ if and only if it is closed under skew sums. We say that a set of
permutations $Y$ is \emph{wreath-closed} if $Y\wr Y\subseteq Y$. A principal
permutation class $\Av(\pi)$ is wreath-closed if and only if $\pi$ is a simple
permutation, and more generally, a permutation class is wreath-closed if and
only if it is equal to $\Av(F)$ for a set $F$ that only contains simple
permutations~\cite[Proposition 1]{alat_sim}. Simple permutations, inflations and
wreath products are crucial concepts in understanding the structure of
permutation classes, as demonstrated, e.g., by the work of Brignall~\cite{brig}
or Brignall, Huczynska and Vatter~\cite{bhv_sim}.

\subsubsection*{Merging and splitting}
Recall that a permutation $\pi$ is \emph{merged} from permutations $\rho$ and
$\tau$, if it can be partitioned into two disjoint subsequences, one of which
is an occurrence of $\rho$ and the other is an occurrence of~$\tau$.
For two permutation classes $A$ and $B$, we write $A\m B$ for the class of 
all the permutations that can be obtained by merging a (possibly empty)
permutation from $A$ with a (possibly empty) permutation from~$B$. Clearly, $A\m
B$ is again a permutation class.

We are interested in finding out which permutation classes $C$ can be merged
from finitely many proper subclasses. We say that a multiset
$\{P_1,\dotsc,P_m\}$ of permutation classes forms a \emph{splitting} of a
permutation class $C$ if $C\subseteq P_1\m P_2\m\dotsb\m P_m$. The classes $P_i$
are the \emph{parts} of the splitting. The splitting is \emph{nontrivial} if
none of its parts is a superset of $C$, and the splitting is \emph{irredundant}
if no proper submultiset of  $\{P_1,\dotsc,P_m\}$ forms a splitting of~$C$. We
say that a class of permutations is \emph{splittable} if it has a nontrivial
splitting. To familiarize ourselves with this key notion of the paper, we
now provide several equivalent definitions of splittability.

\begin{lemma}\label{lem-simple} For a class $C$ of permutations, the following
properties are equivalent:
\begin{itemize}
 \item[(a)] $C$ is splittable.
 \item[(b)] $C$ has a nontrivial splitting into two parts. 
 \item[(c)] $C$ has a splitting into two parts, in which
each part is a proper subclass of~$C$.
 \item[(d)] $C$ has a nontrivial splitting into two parts, in which
each part is a principal class.
\end{itemize}
\end{lemma}
\begin{proof}
Suppose that $C$ is splittable, and let $\{P_1,P_2, \dotsc,P_k\}$ be a
nontrivial irredundant splitting of~$C$. Define a class $Q=P_2\m\dotsb\m
P_k$. By irredundance, the class $C$ is not a subset of $Q$, and therefore
$\{P_1,Q\}$ is a nontrivial splitting of $C$ into two parts. Moreover,
$\{P_1\cap C,Q\cap C\}$ is a splitting into two parts, with each part being a
proper subclass of~$C$. This shows that (a) implies (b) and~(c). 

To prove that (b) implies (d), suppose that $\{P_1,P_2\}$ is a nontrivial
splitting of $C$ into two parts. Choose arbitrary permutations $\sigma\in
C\setminus P_1$ and $\pi\in C\setminus P_2$. Then $\Av(\sigma)$ is a superset of
$P_1$ but not a superset of $C$, and similarly for $\Av(\pi)$ and $P_2$. Thus
$\{\Av(\sigma), \Av(\pi)\}$ is a splitting of~$C$ that witnesses property~(d).
Since both (c) and (d) clearly imply (a), the statements (a) through (d)
are equivalent. 
\end{proof}

We remark that in general, a splittable class need not have a splitting
that simultaneously satisfies properties (c) and (d) of Lemma~\ref{lem-simple}.
For instance, we will later see that the class $\Av(1342)$ is splittable,
but it cannot be split into its principal proper subclasses, not even if we
consider splittings with an arbitrary finite number of parts.

For a multiset $\cS$ and an integer $k$, let $\mexp{k}{\cS}$ denote the
multiset
obtained by increasing $k$-times the multiplicity of each element of~$\cS$.

\subsubsection*{Specific permutation classes}
We will occasionally refer to several specific permutation classes. One such
class is the class $\Av(\{231,312\})$ whose members are known as \emph{layered
permutations}. Layered permutations are exactly the permutations that can be
obtained as direct sums of decreasing sequences, or equivalently, they are the
elements of $\Av(12)\wr\Av(21)$. Another important class is the class
$\Av(\{2413,3142\}$ of \emph{separable permutations}. Separable permutations are
exactly the permutations obtainable from the singleton permutation by iterated
direct sums and skew sums. Therefore, they form the smallest nonempty class
closed under both direct sums and skew sums.

\subsection{Atomicity, amalgamation and Ramseyness}\label{ssec-ramsey}

Splittability is related to several previously studied structural properties of
permutation classes. These properties, including splittability itself, are 
not specific to classes of permutations, but are more generally applicable to
classes of arbitrary objects that are partially ordered by containment, such as
posets, graphs, directed graphs or uniform hypergraphs. 

A convenient formalism for such structures is based on the notion of
`relational structure'. A \emph{relational structure} with \emph{signature}
$(s_1,\dotsc,s_k)$ on a vertex set $V$ is a $k$-tuple $(R_1,\dotsc,R_k)$, where
$R_i$ is a relation of arity $s_i$ on~$V$, i.e., $R_i$ is a set of ordered
$s_i$-tuples of (not necessarily distinct) elements of~$V$. For a relational
structure, we may define in an obvious way the notion of induced substructure. A
\emph{class of relational structures} is a set of isomorphism types of finite
relational structures all sharing the same signature and closed under taking
induced substructures. A detailed treatment of the topic of relational
structures can be found, e.g., in the papers of Cameron~\cite{cam_age} or
Pouzet~\cite{pouzet}. 

As explained, e.g., in~\cite{cam_per}, a permutation of order $n$ can be
represented as a relational structure formed by two linear orders on a vertex
set of size $n$, and so permutation classes are a particular case of classes of
relational structures of signature $(2,2)$.

A structural property that is directly related to splittability is known as
atomicity. A class of relational structures $C$ is \emph{atomic} if it does not
have proper subclasses $A$ and $B$ such that $C=A\cup B$.
Fra\"iss\'e (\cite{fra54}, see also~\cite{fra86}) has proved the following
result:

\begin{fact}[Fra\"\i ss\'e~\cite{fra54}]
 For a class $C$ of relational structures, the following properties
are equivalent:
\begin{itemize}
 \item $C$ is atomic.
 \item For any two elements $\rho,\tau\in C$ there is a $\sigma\in C$ which
contains both $\rho$ and $\tau$ as substructures (this is known as
\emph{the joint embedding property}, and $\sigma$ is a \emph{joint embedding}
of $\rho$ and $\tau$).
 \item There is a (possibly infinite) relational structure $\Gamma$ such that
$C$ is the set of all the finite substructures of $\Gamma$ (we then say that $C$
is \emph{the age} of $\Gamma$).
\end{itemize}
\end{fact}

A permutation class that is not atomic is clearly splittable. Thus, when
studying splittability, we mostly focus on atomic classes. For such situation,
we may slightly extend Lemma~\ref{lem-simple}.

\begin{lemma}\label{lem-simple2}
If $C$ is an atomic permutation class, then the following properties are 
equivalent:
\begin{itemize}
 \item[(a)] $C$ is splittable.
 \item[(b')] $C$ has a nontrivial splitting into two equal parts, i.e., a
splitting of the form $\{P,P\}$.
 \item[(c')] $C$ has a splitting of the form $\{P,P\}$, where $P$ is a
proper subclass of~$C$.
 \item[(d')] $C$ has a nontrivial splitting of the form $\{P,P\}$, where $P$ is
a principal class.
\end{itemize}
\end{lemma}
\begin{proof}
Let $C$ be a splittable atomic class. By part (d) of Lemma~\ref{lem-simple}, $C$
has a nontrivial splitting of the form $\{\Av(\pi),\Av(\pi')\}$. Note that both
$\pi$ and $\pi'$ belong to $C$, otherwise the splitting would be trivial. By
the joint embedding property, there is a permutation $\sigma\in C$ which
contains both $\pi$ and~$\pi'$. Consequently, $\Av(\pi)$ and $\Av(\pi')$ are
subclasses of $\Av(\sigma)$. This shows that $\{\Av(\sigma),\Av(\sigma)\}$ is a
splitting of $C$ witnessing property (d'), and $\{\Av(\sigma)\cap
C,\Av(\sigma)\cap C\}$ is witnessing property (c'). The rest of the lemma
follows trivially.
\end{proof}

It is not hard to see that any principal class $\Av(\pi)$ of permutations
is atomic. Indeed, if $\pi$ is indecomposable, then $\Av(\pi)$ is closed under
direct sums, and otherwise it is closed under skew sums, which in both cases
implies the joint embedding property. 

Atomic permutation classes have been studied by Atkinson, Murphy and Ru\v
skuc~\cite{amr_wqo,amr_sub,murphy_thesis}. Among other results, they
show (see \cite[Theorem 2.2]{amr_wqo} and \cite[Proposition 188]{murphy_thesis})
that every permutation class that is partially well-ordered by containment
(equivalently, any class not containing an infinite antichain) is a union of
finitely many atomic classes. We remark that such classes of relational
structures admit the following Dilworth-like characterization~\cite[Theorem
1.6]{jelkla}: a class $C$ of relational structures is a union of $k$ atomic
classes if and only if it does not contain $k+1$ elements no two of which admit
a joint embedding in $C$. Moreover, a class $C$ that is not a union of finitely
many atomic classes contains an infinite sequence of elements, no two of which
admit a joint embedding in~$C$.

The joint embedding property can be further strengthened, leading to a concept
of amalgamation, which we define here for permutations, though it can be
directly extended to other relational structures. Informally speaking, a
permutation class $C$ is $\pi$-amalgamable if for any two permutations $\rho_1$
and $\rho_2$ of $C$, each having a prescribed occurrence of $\pi$, we may find a
joint embedding of $\rho_1$ and $\rho_2$ in which the two prescribed occurrences
of $\pi$ coincide. Formally, suppose that $C$ is a permutation class, and
$\pi$ is its element. Then $C$ is \emph{$\pi$-amalgamable} if for any two
permutations $\rho_1,\rho_2\in C$ and any two mappings $f_1$ and $f_2$ such that
$f_i$ is an embedding of $\pi$ into $\rho_i$, there is a permutation $\sigma\in
C$ and two embeddings $g_1$ and $g_2$, where $g_i$ is an embedding of $\rho_i$
into $\sigma$, with the property that $g_1\circ f_1=g_2\circ f_2$.

Let $\binom{\sigma}{\pi}$ denote the set of all occurrences of a permutation
$\pi$ in a permutation $\sigma$, and if $S$ is a subsequence of $\sigma$, let
$\binom{S}{\pi}$ be the set of those occurrences of $\pi$ that are contained
in~$S$. A permutation class $C$ is \emph{$\pi$-Ramsey} if for every $\rho\in C$
there is a $\sigma\in C$ such that whenever we color $\binom{\sigma}{\pi}$ by
two colors, there is a subsequence $S\in\binom{\sigma}{\rho}$ such that all
elements of $\binom{S}{\pi}$ have the same color.

The next lemma follows from the results of Nešetřil~\cite{nese_graf} (see
also~\cite[Theorem 4.2]{nese_hom}) 
\begin{fact}\label{fac-ama} Let $C$ be an atomic permutation class and $\pi$ an
element of~$C$. If $C$ is $\pi$-Ramsey then $C$ is also $\pi$-amalgamable. 
\end{fact}

We say, for $k\in \bbN$, that a permutation class $C$ is \emph{$k$-Ramsey} if it
is $\pi$-Ramsey for every $\pi\in C$ of order at most $k$, and we say that it is
\emph{$k$-amalgamable} if it is $\pi$-amalgamable for every~$\pi\in C$ of order
at most~$k$. We also say that $C$ is a \emph{Ramsey class} (or \emph{amalgamable
class}) if it is $k$-Ramsey (or $k$-amalgamable) for every~$k$. 

Ramsey classes and amalgamable classes of various types of relational structures
have attracted considerable amount of attention, due in part to their connection
to so-called Fra\"\i ssé limits and homogeneous structures. We shall not explain
these concepts here, and refer the interested reader to, e.g., a survey of the
field presented in~\cite{nese_hom}. 

Cameron~\cite{cam_per} has shown that there are only five infinite amalgamable
classes of permutations: these are $\Av(12)$, $\Av(21)$, $\Av(231,312)$ (the
class of layered permutations), $\Av(213,132)$ (the class of complements of
layered permutations), and the class of all permutations. B\"ottcher and
Foniok~\cite{bofo} have subsequently proved that all these amalgamable classes
are Ramsey. In view of Fact~\ref{fac-ama}, there can be no other atomic Ramsey
classes of permutations. By suitably adapting Cameron's proof, it is actually
possible to deduce that the five nontrivial amalgamable classes are also the
only 3-amalgamable permutation classes, that is, any 3-amalgamable class of
permutations is already Ramsey.

The above-defined Ramsey properties are closely related to splittability: it is
straightforward to observe, by referring to condition (d') of
Lemma~\ref{lem-simple2}, that an atomic class of permutations is unsplittable 
precisely when it is 1-Ramsey. Consequently, by Fact~\ref{fac-ama}, any
permutation class that fails to be 1-amalgamable is splittable. To make our
exposition self-contained, we prove this simple result here.

\begin{lemma}\label{lem-ama}
 If a permutation class $C$ is not 1-amalgamable, then it is splittable. More
precisely, if $\rho_1$ and $\rho_2$ are two elements of $C$ that fail to have a
1-amalgamation in $C$, then $C$ has the splitting
$\{\Av(\rho_1),\Av(\rho_2)\}$. \end{lemma}
\begin{proof}
 Suppose that $C$ is not $1$-amalgamable. That is, there are two permutations 
$\rho_1\in C$ and $\rho_2\in C$ of size $n$ and $m$ respectively, and two
embeddings $f_1$ and $f_2$ of the singleton permutation 1 into $\rho_1$ and
$\rho_2$, such that there does not exist any $\sigma\in C$ with embeddings $g_i$
of $\rho_i$ into $\sigma$ that would satisfy $g_1\circ f_1= g_2\circ f_2$.

Suppose that $\rho_1(a)$ is the unique element of $\rho_1$ in the range of
$f_1$, and similarly $\rho_2(b)$ the unique element of $\rho_2$ in the range
of~$f_2$. Let $\sigma\in C$ be arbitrary. Our goal is to color the elements of
$\sigma$ red and blue, so that the red elements avoid $\rho_1$ and the blue
ones avoid~$\rho_2$. To achieve this, we color an element $\sigma(i)$ of
$\sigma$ blue if and only if there is an embedding of $\rho_1$ into $\sigma$
which maps $\rho_1(a)$ to $\sigma(i)$. The remaining elements of $\sigma$ are
red. 

We see that the red elements of $\sigma$ do not contain $\rho_1$, since
otherwise there would be an embedding of $\rho_1$ into $\sigma$ that maps
$\rho_1(a)$ to a red element. Also, the blue elements of $\sigma$ have no copy
of $\rho_2$, and more generally, in any embedding of $\rho_2$ into $\sigma$, the
element $\rho_2(b)$ must map to a red element of $\sigma$, since otherwise we
would obtain a joint embedding of $\rho_1$ and $\rho_2$ identifying $\rho_1(a)$
with $\rho_2(b)$, which is impossible.
\end{proof}

Concepts analogous to splittability and 1-Ramseyness have been previously
studied, under various names, in connection to combinatorial structures other
than permutations. This line of research dates back at least to the work of
Folkman~\cite{folk}, who considered both edge-decompositions and
vertex-decompositions of graphs of bounded clique size, and showed, among other
results, that the class of graphs avoiding a clique of a given size is
1-Ramsey. Nešetřil and R\"odl~\cite{nero} obtained other examples of
1-Ramsey classes of graphs. Later, Pouzet~\cite{pou_imp} and El-Zahar and
Sauer~\cite{elsa} considered a notion equivalent to splittability in context of
atomic classes of relational structures, as part of a hierarchy of several
Ramsey-type properties. For further developments in this area, see e.g. the
works of Laflamme et al.~\cite{LNPS_indivisible} or Bonato et
al.~\cite{BCDT_pigeonhole}. We will not go into any further details of these
results, as they do not seem to be applicable to our problem of identifying
splittable classes of permutations.

In Section~\ref{sec-unsplit}, we will give unsplittability criteria, which will
imply, among other results, that any wreath-closed permutation class is
unsplittable. Next, in Section~\ref{sec-split}, we will look for examples of 
splittable classes. This turns out to be considerably more challenging. Our main
result in this direction shows that the class $\Av(\sigma)$ is splittable
whenever $\sigma$ is a decomposable permutation of order at least four. We then
describe, in Subsection~\ref{ssec-circ}, the connection between splittability of
permutation classes and the chromatic number of circle graphs. This connection
allows us to exploit previous work to give both positive and negative results on
the existence of certain permutation splittings.

\section{Unsplittable classes and unavoidable patterns}\label{sec-unsplit}

We now focus on the unsplittable permutation classes. By Lemma~\ref{lem-simple},
when looking for splittings of a class $C$, we may restrict our attention to
splittings whose parts are principal classes not containing $C$, i.e., classes
of the form $\Av(\pi)$ for some~$\pi\in C$. The basic idea of our approach will
be to identify a large set of permutations $\pi\in C$ for which we can prove
that $\Av(\pi)$ is not a part of any irredundant splitting of~$C$. This
motivates our next definition.

\begin{definition}
Let $C$ be a permutation class, and let $\pi$ be an element of~$C$. We say that
$\pi$ is \emph{unavoidable in $C$}, if $C$ has no irredundant splitting that
contains $\Av(\pi)$ as a part. We let $\U C$ denote the set of all
the unavoidable permutations in~$C$.
\end{definition}

The next two observations list several basic properties of unavoidable
permutations. These properties follow directly from the definitions or from the
arguments used to prove Lemma~\ref{lem-simple}. We therefore omit their proofs.

\begin{observation}\label{obs-unavo}
For a permutation class $C$ and a permutation $\pi\in C$, the following
statements are equivalent.
\begin{enumerate}
\item The permutation $\pi$ is unavoidable in $C$.
\item For any permutation $\tau\in C$, there is a permutation $\sigma\in C$
such that any red-blue coloring of $\sigma$ has a red copy of $\tau$ or a
blue copy of~$\pi$.
\item In any irredundant splitting $\{P_1,\dotsc,P_k\}$ of $C$, all the 
parts $P_i$ contain~$\pi$.
\item $C$ has no nontrivial splitting into two parts, where one of the parts
is~$\Av(\pi)$.
\end{enumerate}
\end{observation}

\begin{observation}\label{obs-core}
The set $\U C$ of unavoidable permutations of a nonempty permutation class $C$
has these properties:
\begin{enumerate}
\item $\U C$ is a nonempty permutation class, and in particular, it is
hereditary.
\item $\U C \subseteq C$.                                              
\item If $\{P_1,\dotsc,P_m\}$ is an irredundant splitting of $C$, then $\U
C\subseteq P_i$ for each~$i$.
\item $\U C =C$ if and only if $C$ is unsplittable.               
\end{enumerate} 
\end{observation}

By the last part of the previous observation, to show that a permutation class
$C$ is unsplittable, it is enough to prove $\U C=C$. To achieve this, we will
show that certain closure properties of $C$ imply analogous closure properties
of~$\U C$.

\begin{lemma} Let $C$ be a permutation class. If, for a set of permutations 
$X$, the class $C$ is closed under $\wr X$, then $\U C$ is also closed under 
$\wr X$, and if $C$ is closed under $X\wr$, then so is $\U C$. Consequently, 
if $C$ is wreath-closed, then $\U C = C$ and $C$ is unsplittable.
\end{lemma}
\begin{proof}              
We first prove that if $C$ is closed under $\wr X$, then so is $\U C$. 
Suppose $C$ is closed under $\wr X$. Note that we may assume that $X$ itself is
wreath-closed; this is because $\wr$ is associative, so if $C$ is closed under
$\wr X$, it is also closed under $\wr(X\wr X)$ and therefore it is closed under
$\wr Y$ where $Y$ is the wreath-closure of~$X$. 

Choose a permutation $\pi\in\U C$ of order $k$, and $k$ permutations 
$\rho_1,\dotsc,\rho_k\in X$. We wish to prove that 
$\pi[\rho_1,\dotsc,\rho_k]$ belongs to $\U C$. Without loss of generality, 
we assume that all $\rho_i$ are equal to a single permutation $\rho$; if 
not, we simply put $\rho\in X$ to be a permutation which contains all the
$\rho_i$ (such a $\rho$ exists, since $X$ is wreath-closed) and prove the
stronger fact that $\pi[\rho,\dotsc,\rho]$ belongs to $\U C$. 

Let us use the notation $\pi[\rho]$ as shorthand for $\pi[\rho,\dotsc,\rho]$. So
our goal now reduces to showing that $\pi[\rho]$ belongs to $\U C$ for every
$\pi\in\U C$ and $\rho\in X$. We base our argument on the second equivalent
definition from Observation~\ref{obs-unavo}. Fix a permutation
$\tau\in C$. We want to find a permutation $\sigma\in C$ such that each red-blue
coloring of $\sigma$ either contains a red copy of $\tau$ or a blue copy of 
$\pi[\rho]$. We already know that $\pi$ belongs to $\U C$, so there
is a permutation $\sigma'$ such that each red-blue coloring of $\sigma'$ has
either a red copy of $\tau$ or a blue copy of~$\pi$. Let $\ell$ be the 
order of $\sigma'$.

Define $\sigma$ by $\sigma=\sigma'[\rho]$, and view $\sigma$ as a concatenation
of $\ell$ blocks, each being a copy of~$\rho$. Fix an arbitrary red-blue
coloring of $\sigma$. We now define a red-blue coloring of $\sigma'$ as follows:
an element $\sigma'_i$ of $\sigma'$ is red if the $i$-th block in $\sigma$
has at least one red point, otherwise it is blue. Note that this coloring of
$\sigma'$ has the property that if $\sigma'$ contains a red copy of any pattern
$\beta$ then $\sigma$ also contains a red copy of $\beta$, and if $\sigma'$
contains a blue copy of $\beta$, then $\sigma$ has a blue copy of $\beta[\rho]$.
In particular, $\sigma$ either contains a red copy of $\tau$ or a blue
copy of~$\pi[\rho]$. This proves that $\pi[\rho]$ belongs to $\U C$, as
claimed.

We now show that if $C$ is closed under $X\wr$ then so is~$\U C$. Fix 
a permutation $\rho\in X$ of order $k$, and a $k$-tuple $\pi_1,\dotsc,\pi_k$
of permutations from $\U C$. We will show that $\rho[\pi_1,\dotsc,\pi_k]$
belongs to~$\U C$. 

Fix a permutation $\tau\in C$. Since $\pi_i$ is in $\U C$, there is a
permutation $\sigma_i\in C$ whose every red-blue coloring has either a red copy
of $\tau$ or a blue copy of~$\pi_i$. Define
$\sigma=\rho[\sigma_1,\dotsc,\sigma_k]$, viewing it as a union of $k$ blocks,
with the $i$-th block being a copy of~$\sigma_i$. Fix a red-blue coloring
of~$\sigma$. The $i$-th block of $\sigma$ either contains a red copy of $\tau$
or a blue copy of~$\pi_i$. Thus, if $\sigma$ has no red copy of $\tau$, it must
have a blue copy of~$\rho[\pi_1,\dotsc,\pi_k]$, showing that
$\rho[\pi_1,\dotsc,\pi_k]$ belongs to~$\U C$.

It remains to show that if $C$ is wreath-closed then $\U C=C$. But if $C$ is 
wreath-closed, then it is closed under $\wr C$, so $\U C$ is also closed under 
$\wr C$. But since $\U C$ is a nonempty subclass of $C$ by
Observation~\ref{obs-core}, this means that $\U C=C$. \end{proof}

\begin{corollary}\label{cor-indeco}
 Every wreath-closed permutation class is unsplittable. In particular, if
$\pi$ is a simple permutation then $\Av(\pi)$ is unsplittable.
\end{corollary}

Not all unsplittable classes are wreath-closed. For instance, the class $C$ of
layered permutations is unsplittable; to see this, note that $C$ is closed under
$12\wr$ as well as under $\wr21$, and it is the smallest nonempty class with
these properties. Since $\U C$ has the same closure properties, we must have $\U
C=C$.

There is even an example of a principal class that is unsplittable even though
it is not wreath-closed, namely the class $\Av(132)$. To show that this class is
indeed unsplittable we need a more elaborate argument.

\begin{lemma}\label{lem-1p} Let $\rho$ be an indecomposable permutation.
 Let $C$ be the class $\Av(1\oplus\rho)$. Then, for any two permutations
$\pi$, $\pi'$ such that $\pi\in\U C$ and $1\oplus\pi'\in\U C$, 
we have $\pi\oplus\pi'\in\U C$.
\end{lemma}
\begin{proof} Note that if $\sigma$ and $\sigma'$ are two permutations avoiding
the pattern $1\oplus \rho$, and if $\sigma''$ is obtained by inflating any
LR-minimum of $\sigma'$ by a copy of $\sigma$, then $\sigma''$ also
avoids~$1\oplus\rho$. This follows easily from the fact that $\rho$ is
indecomposable.

Let $\tau$ be any element of $C$. Our goal is to find a permutation $\sigma''\in
C$ whose every red-blue coloring has a red copy of $\tau$ or a blue copy of
$\pi\oplus\pi'$. Since $\pi\in \U C$, there is a $\sigma\in C$ such that 
every red-blue coloring of $\sigma$ has a red copy of $\tau$ or a blue copy of
$\pi$. Similarly, there is a $\sigma'\in C$ whose every red-blue coloring has a
red copy of $\tau$ or a blue copy of $1\oplus\pi'$. 

Let $\sigma''$ be the permutation obtained by inflating each LR-minimum of
$\sigma'$ by a copy of $\sigma$. Fix any red-blue coloring of $\sigma''$ that
has no red copy of $\tau$. Then every $\sigma$-block in $\sigma''$ contains a
blue copy of $\pi$. Consider a two-coloring of $\sigma'$ in which every
LR-minimum is blue, and all the other elements have the same color as the
corresponding elements in~$\sigma''$. This coloring contains a blue copy of
$1\oplus\pi'$. This means that $\sigma''$ has a blue copy of $\pi'$ which is
disjoint from all the $\sigma$-blocks obtained by inflating the LR-minima of
$\sigma'$. Combining this blue copy of $\pi'$ with a blue copy of $\pi$ in an
appropriate $\sigma$-block, we get a blue copy of $\pi\oplus\pi'$ in~$\sigma''$.
\end{proof}

\begin{proposition}\label{pro-132}
The class $\Av(132)$ is unsplittable.
\end{proposition}
\begin{proof} 
Let $C=\Av(132)$. We will show that $C$ is equal to $\U C$.
Pick a permutation $\pi\in C$, with $n\ge 2$.
Let $k$ be the index such that $\pi(k)=n$. If $k=n$, then $\pi$ can be written
as $\pi'\oplus 1$ for some $\pi'\in C$. If $k<n$, then all the elements
$\pi(k+1),\dotsc,\pi(n)$ are smaller than any element in $\pi(1),\dotsc,\pi(k)$,
otherwise we would find an occurrence of $132$. Consequently, $\pi$ can be
written as $\pi'\ominus\pi''$ for some $\pi', \pi''\in C$.

To show that every $\pi\in C$ belongs to $\U C$, proceed by induction. For
$\pi=1$ this is clear, for $\pi=12$, this follows from the fact that $C$ is
$\wr12$-closed. If $\pi$ is equal to $\pi'\ominus\pi''$, use the fact that $C$
is $\ominus$-closed. If $\pi=\pi'\oplus 1$, use Lemma~\ref{lem-1p} and the fact
that $12\in \U C$. 
\end{proof}

\section{Splittable classes and decomposable patterns}\label{sec-split}
We now focus on splittable permutation classes. We are again mostly
interested in principal classes. Let us begin by stating the main result of this
section.

\begin{theorem}\label{thm-split}
If $\pi$ is a decomposable permutation other than 12, 213 or 132, then
$\Av(\pi)$ is a splittable class. 
\end{theorem}

The exclusion of 12, 213 and 132 in the statement of the theorem is necessary,
since it follows from the results of the previous section that $\Av(12)$,
$\Av(213)$ and $\Av(132)$ are unsplittable.

As the first step towards the proof of Theorem~\ref{thm-split},
we deal with patterns that are decomposable into (at least) three
parts. 

\begin{proposition}\label{pro-cjs}
 Let $\alpha$, $\beta$ and $\gamma$ be three nonempty permutations. The class
$\Av(\alpha\oplus\beta\oplus\gamma)$ is splittable, and more precisely, it
satisfies
\[
 \Av(\alpha\oplus\beta\oplus\gamma)\subseteq\Av(\alpha\oplus\beta)\m
\Av(\beta\oplus\gamma).
\]
\end{proposition}

We note that a weaker version of Proposition~\ref{pro-cjs} (with an extra
assumption that $\beta$ has the form $\beta'\ominus 1$ for some $\beta'$) has
been recently proved by Claesson, Jelínek and Steingr\'imsson~\cite[Theorem
3]{cjs}. The proof we present below is actually a simple adaptation of the
argument from~\cite{cjs}.

\begin{proof}[Proof of Proposition~\ref{pro-cjs}] Fix a permutation $\pi$
avoiding the pattern $\alpha\oplus\beta\oplus\gamma$. We will color the elements
of $\pi$ red and blue, so that the red elements will avoid $\alpha\oplus\beta$
and the blue ones will avoid $\beta\oplus\gamma$. We construct the coloring by
taking the elements $\pi(1)$ to $\pi(n)$ successively, and having colored the
elements $\pi(1),\dotsc,\pi(i-1)$ for some $i$, we determine the color of
$\pi(i)$ by these rules:
\begin{itemize}
 \item If coloring $\pi(i)$ red completes a red occurrence of
$\alpha\oplus\beta$, color $\pi(i)$ blue.
 \item If for some $j<i$ the element $\pi(j)$ is blue and $\pi(j)<\pi(i)$, color
$\pi(i)$ blue.
 \item Otherwise color $\pi(i)$ red.
\end{itemize}
Clearly, the coloring determined by these rules avoids a red copy of
$\alpha\oplus\beta$. We now show that it also avoids a blue copy of
$\beta\oplus\gamma$. Suppose for contradiction that $\pi$ has a blue occurrence
of $\beta\oplus\gamma$, and let $\beta_B$ and $\gamma_B$ denote the occurrences
of $\beta$ and $\gamma$ in this blue occurrence of $\beta\oplus\gamma$, with
$\beta_B$ being completely to the left and below~$\gamma_B$. 

Let $\pi(a)$ be the smallest element of $\beta_B$. Since $\pi(a)$ is blue, $\pi$
must have a blue element $\pi(b)$ such that $b\le a$, $\pi(b)\le \pi(a)$, and
changing the color of $\pi(b)$ from blue to red would create a red copy of
$\alpha\oplus\beta$ in the sequence $\pi(1),\pi(2),\dotsc,\pi(b)$. In other
words,
$\pi$ has a copy of $\alpha\oplus\beta$ whose rightmost element is $\pi(b)$ and
whose remaining elements are all red. Let $\alpha_R$ and $\beta_R$ denote the
copies of $\alpha$ and $\beta$ in this copy of $\alpha\oplus\beta$, with
$\alpha_R$ to the left and below $\beta_R$. 

Note that every element of $\alpha_R$ is smaller than $\pi(b)$, and 
therefore every element of $\alpha_R$ is smaller than all the elements
of~$\beta_B$. Let $\pi(c)$ be the leftmost element of $\beta_B$.
If every element of $\alpha_R$ is to the left of $\pi(c)$, then
$\alpha_R\cup\beta_B\cup\gamma_B$ is a copy of $\alpha\oplus\beta\oplus\gamma$
in $\pi$, which is impossible. Thus at least one element of $\alpha_R$ is to the
right of $\pi(c)$, and consequently, all the elements of $\beta_R$ are to the
right of~$\pi(c)$. It follows that all the red elements in $\beta_R$ are smaller
than $\pi(c)$, for otherwise they would be colored blue by the second rule of
our coloring. This means that $\alpha_R\cup\beta_R\cup\gamma_B$ is a copy of
$\alpha\oplus\beta\oplus\gamma$, a contradiction.
\end{proof}

\begin{corollary}\label{cor-easysplit}
 If $\alpha$ and $\beta$ are permutations of order at least two, then the class
$\Av(\alpha\oplus\beta)$ is splittable.
\end{corollary}
\begin{proof}
 We see that $\Av(\alpha\oplus\beta)\subseteq\Av(\alpha\oplus 1\oplus\beta)$,
and Proposition~\ref{pro-cjs} shows that 
\[
 \Av(\alpha\oplus 1\oplus\beta)\subseteq\Av(\alpha\oplus 1)\m\Av(1\oplus\beta).
\]
Therefore, $\Av(\alpha\oplus\beta)$ admits the splitting $\{\Av(\alpha\oplus
1),\Av(1\oplus\beta)\}$.
\end{proof}

To prove Theorem~\ref{thm-split}, it remains to deal with the classes
$\Av(\alpha\oplus\beta)$, where $\alpha$ or $\beta$ has order one. As the two
cases are symmetric, we may assume that $\alpha=1$. We may also assume that
$\beta$ is indecomposable, because otherwise $\Av(1\oplus\beta)$
is splittable by Proposition~\ref{pro-cjs}. Finally, we may assume that $\beta$
has order at least three, since we already know that $\Av(12)$ and $\Av(132)$
are unsplittable. Let us therefore focus on the classes $\Av(1\oplus\sigma)$,
where $\sigma$ is an indecomposable permutation of order at least three. 

To handle these
`hard' cases of Theorem~\ref{thm-split}, we will introduce the notion of ordered
matchings, and study the splittings of hereditary classes of matchings.
Matchings are more general structures than permutations, in the sense that the
containment poset of permutations is a subposet of the containment poset of
matchings. Moreover, the arguments we use in our proof can be more naturally
presented in the terminology of matchings, rather than permutations. The
downside is that we will have to introduce a lot of basic terminology related
to matchings.

In Subsection~\ref{ssec-match}, we will introduce matchings and describe how
they relate to permutations. Next, in Subsection~\ref{ssec-thmsplit}, we present
the proof of Theorem~\ref{thm-split}. Although the proof is constructive, the
splittings we construct in the proof involve classes defined by avoidance of
rather large patterns. Such splittings do not seem to reveal much information
about the classes being split. For this reason, Subsection~\ref{ssec-expl}
gives another splitting algorithm, which is less general, but which provides
more natural splittings involving avoiders of small patterns. Using this
alternative splitting approach, we establish, in Subsection~\ref{ssec-circ}, an
equivalence between the existence of certain permutation splittings and the
colorability of circle graphs of bounded clique size. This allows us to use
previous bounds on the chromatic number of clique-avoiding circle graphs to
deduce both positive and negative results about existence of permutation
splittings.

\subsection{Ordered matchings}\label{ssec-match}

Let $P=\{x_1<x_2<\dotsb<x_{2n}\}$ be a set of $2n$ real numbers, represented by points
on the real line. A \emph{matching} (or, more properly, an \emph{ordered
perfect matching}) on the point set $P$ is a set $M$ of pairs of points from~$P$,
such that every point of $P$ belongs to exactly one pair from~$M$. The
elements of $M$ are the \emph{arcs} of the matching $M$, while the elements of
$P$ are the \emph{endpoints} of~$M$. 

We represent an arc $\alpha$ of a matching $M$ by an ordered pair $(a,b)$ of
points, and we make the convention that $a$ is left of~$b$. The points $a$ and
$b$ are referred to as the \emph{left endpoint} and \emph{right endpoint} of
$\alpha$, respectively, and denoted by $\lend\alpha$ and $\rend\alpha$.
We visualize the arcs as half-circles connecting the two
endpoints, and situated in the upper half-plane above the real
line.

If $M$ is a matching on a point set $\{x_1<x_2<\dotsb<x_{2n}\}$ and  $N$ is a
matching on a point set $P'=\{y_1<y_2<\dotsb<y_{2n}\}$, we say that $M$ and $N$
are \emph{isomorphic}, written as $M\cong N$, if for every $i,j$ we have the
equivalence $(x_i,x_j)\in M \iff (y_i,y_j)\in N$. A matching $M$ \emph{contains}
a matching $N$, if it has a subset $M'\subseteq M$ such that the matching $M'$
is isomorphic to~$N$. If $M$ does not contain $N$, we say that $M$ \emph{avoids}
$N$. Let $\MAv(N)$ be the set of the isomorphism classes of matchings that
avoid~$N$.

Let $M$ be a matching, and let $\alpha=(a,b)$ and $\beta=(c,d)$ be two arcs
of~$M$. We say that \emph{$\alpha$ crosses $\beta$ from the left} if $a<c<b<d$,
and we say that \emph{$\alpha$ is nested below $\beta$} if $c<a<b<d$. We say
that a point $x\in\bbR$ is \emph{nested below the arc $\alpha$} if $a<x<b$
(note that $x$ does not have to be an endpoint of $M$). We say that two arcs 
are \emph{in series} if they neither cross nor nest, which means that one of
them is completely to the left of the other. We say that an arc $\alpha$ of a
matching $M$ is \emph{short} if its endpoints are adjacent, i.e., if there is
no endpoint of $M$ nested below~$\alpha$. An arc that is not short is
\emph{long}.

We define the \emph{directed intersection graph} of $M$, denoted by $\dint M$,
to be the graph whose vertices are the arcs of $M$, and $\dint M$ has a directed
edge from $\alpha$ to $\beta$ if $\alpha$ crosses $\beta$ from the left. The
\emph{intersection graph} of $M$, denoted by $\uint M$, is the graph
obtained from $\dint M$ by omitting the orientation of its edges. Note that two
arcs of $M$ are adjacent in $\uint M$ if and only if the half-circles
representing the two arcs intersect. We say that a matching is \emph{connected}
if its intersection graph is a connected graph. 

Let us remark that the graphs that arise as intersection graphs of matchings are
known as \emph{circle graphs} in graph theory literature. We refer the reader to
the surveys~\cite{brand,spinrad} for more information on this graph class.

Let $M$ and $N$ be matchings with $m$ and $n$ arcs, respectively. We let
$M\uplus N$ denote the matching $Q$ which is a disjoint union of two matchings
$Q_1$ and $Q_2$, such that $Q_1\cong M$, $Q_2\cong N$, and any endpoint of $Q_1$
is to the left of any endpoint of~$Q_2$. This determines $M\uplus N$ uniquely up
to isomorphism. We say that a matching $M$ is \emph{$\uplus$-indecomposable} if
it cannot be written as $M\cong M_1\uplus M_2$ for some nonempty matchings $M_1$
and $M_2$. Note that a connected matching is $\uplus$-indecomposable, but the
converse is not true in general. Any matching $M$ can be uniquely written as
$M\cong M_1\uplus M_2\uplus\dotsb\uplus M_k$ where each $M_i$ is a nonempty
$\uplus$-indecomposable matching. We call the matchings $M_i$ the \emph{blocks}
of~$M$.

We say that a matching $M$ is \emph{merged} from two matchings $M_1$ and $M_2$,
if the arcs of $M$ can be colored red and blue so that the red arcs form a
matching isomorphic to $M_1$ and the blue ones are isomorphic to~$M_2$. Given
this concept of merging, we may speak of splittability of matching classes in
the same way as we do in the case of permutations.

\begin{figure}
 \hfil\includegraphics[width=\textwidth]{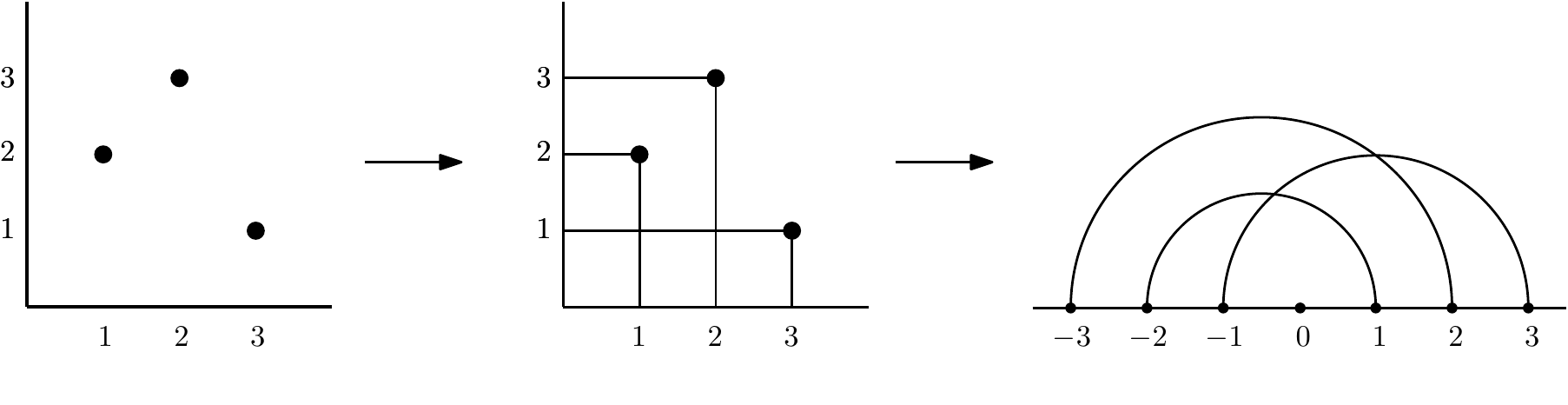}
 \caption{Transforming a permutation $\pi=231$ into the matching
$\ma\pi$.}\label{fig-mpi}
\end{figure}

We now introduce two distinct ways in which a permutation can be encoded by a
matching. Let $\pi$ be a permutation. We let $\ma\pi$ denote the matching on the
point set $\{-n,-n+1,\dotsc,-1\}\cup\{1,2,\dotsc,n\}$ which contains the arc of
the form $(-\pi(i),i)$ for each $i\in[n]$. As shown in Figure~\ref{fig-mpi}, we
may visualize the permutation matching $\ma\pi$ by taking the diagram of $\pi$,
connecting each point of the diagram to the two coordinate axes by a horizontal
and vertical segment, and then deforming the figure so that the pairs of segments
become half-circles. Clearly, the matching $\ma\pi$ is $\uplus$-indecomposable.
It can also be easily verified that $\ma\pi$ is connected if and only if the
permutation $\pi$ is indecomposable. If $C$ is a set of permutations, we let
$\ma C$ denote the set $\{\ma\pi\;|\; \pi\in C\}$. 

We say that a matching $M$ is a \emph{permutation matching} if $M\cong \ma\pi$
for a permutation~$\pi$. 

\begin{observation} For a matching $M$, the following properties are equivalent:
\begin{enumerate}
\item $M$ is a permutation matching.
\item Any left endpoint of $M$ is to the right of any right endpoint.
\item There is a point $x\in\bbR$ nested below all the arcs of~$M$.
\item $M$ has no two arcs in series.
\end{enumerate}
\end{observation}

Permutation matchings form a hereditary class within the class
of all matchings.

\begin{observation}\label{obs-match}
A permutation $\pi$ contains a permutation $\sigma$ if and only if the matching
$\ma\pi$ contains $\ma\sigma$. A permutation $\pi$ can be merged from
permutations $\sigma_1, \dotsc,\sigma_k$ if and only if the matching $\ma\pi$
can be merged
from matchings $\ma\sigma_1,\dotsc,\ma\sigma_k$. Hence,
$\Av(\pi)$ is splittable if and only if $\ma{\Av(\pi)}$~is.
\end{observation}

We now introduce another, less straightforward way to encode a permutation by a
matching.

\begin{figure}
\hfil \includegraphics[width=\textwidth]{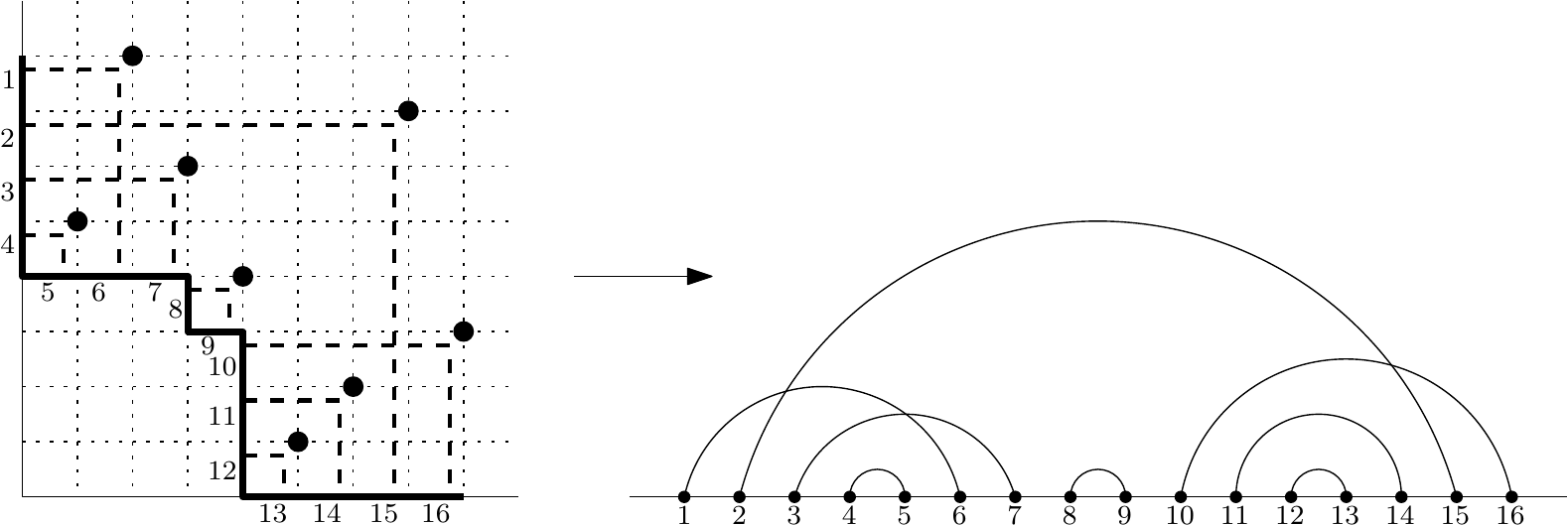}
\caption{The diagram of the permutation $\pi =58641273$, together with its
envelope $P$, represented as the thick line with numbered steps. The dashed
lines represent the arcs of the envelope
matching $E(\pi)$.}\label{fig-envelope}
\end{figure}

\begin{definition}
 Let $\pi$ be a permutation of order $n$. The \emph{envelope} of $\pi$ is a
directed lattice path $P_\pi$ with these properties (see
Fig.~\ref{fig-envelope}):
\begin{itemize}
 \item $P_\pi$ connects the point $(0,n)$ to the point $(n,0)$,
\item every step of $P_\pi$ is either a \emph{down-step} connecting a point
$(i,j)$ to a point $(i,j-1)$, or a \emph{right-step} connecting $(i,j)$ with
$(i+1,j)$, and
\item for every $i\in [n]$, the path $P_\pi$ contains the right step
$(i-1,j)\to(i,j)$ where $j=\min\{\pi(1),\dotsc,\pi(i)\}-1$.
\end{itemize}
\end{definition}
 
Note that all the points of the form $(i,\pi(i))$ are strictly above the
envelope $P_\pi$, and $P_\pi$ is the highest non-increasing lattice path with
this property. For any $i\in[n]$, the \emph{down-step in row $i$} of
$P_\pi$ is the (unique) down-step of the form $(k,i)\to(k,i-1)$ for some $k$,
and the \emph{right-step in column $i$} is the right-step of the form
$(i-1,k)\to(i,k)$.

Note also that $P_\pi$ contains a down-step $(i-1,j)\to(i-1,j-1)$ followed by a
right-step $(i-1,j-1)\to(i,j-1)$ if and only if $\pi(i)$ is an LR-minimum of
$\pi$ and is equal to~$j$.

\begin{definition}\label{def-ematch}
Let $\pi$ be a permutation with envelope $P_\pi$. The \emph{envelope
matching} of $\pi$, denoted by $E(\pi)$, is the matching on the point set
$[2n]$ determined by these rules:
\begin{enumerate}
 \item Label the steps of $P_\pi$ by $\{1,2,\dotsc,2n\}$, in the order in which
they are encountered when $P_\pi$ is traversed from $(0,n)$ to $(n,0)$. 
 \item For every $i$ and $j$ such that $\pi(i)=j$, suppose that the down-step in
row $j$ has label $a$ and the right-step in column $i$ has label $b$. Then
the arc $(a,b)$ belongs to~$E(\pi)$. (Note that we must have $a<b$, since the
point $(i,j)$ is above the path $P_\pi$.)
\end{enumerate}
\end{definition}

Note that each arc of $E(\pi)$ corresponds in an obvious way to an element
of~$\pi$, and the short arcs of $E(\pi)$ correspond precisely to the LR-minima
of~$\pi$.

\begin{observation}\label{obs-ematch}
Suppose that $\pi\in S_n$ is a permutation. Let $\pi(i)$ and
$\pi(j)$ be two elements of $\pi$, and let $\alpha_i$ and $\alpha_j$
be the two corresponding arcs of $E(\pi)$. Then $\rend{\alpha_i}<\rend{\alpha_j}$
if and only if $i<j$, and $\lend{\alpha_i}<\lend{\alpha_j}$ if and only if 
$\pi(i)>\pi(j)$. It follows that $\alpha_i$ is nested below $\alpha_j$ if and 
only if $\pi(i)$ covers $\pi(j)$.

Moreover, if $i<j$ and $\pi(i)>\pi(j)$, then the arcs $\alpha_i$ and
$\alpha_j$ are crossing if and only if $\pi$ has an LR-minimum that covers both
$\pi(i)$ and~$\pi(j)$, otherwise they are in series.
\end{observation}

Observation~\ref{obs-ematch} implies that $E(\pi)$
determines $\pi$ uniquely, and if $\sigma$ and $\pi$ are permutations such
that $E(\sigma)$ contains $E(\pi)$, then $\sigma$ contains~$\pi$. Note however,
that the converse of this last fact does not hold in general: the permutation
$132$ contains $21$, while $E(132)=\{(1,5),(2,6),(3,4)\}$ does not contain
$E(21)=\{(1,2),(3,4)\}$.

\begin{lemma}\label{lem-envelope}
 For a matching $M$ on the point set $[2n]$, the following
statements are equivalent.
\begin{enumerate}
 \item There is a permutation $\pi$ such that $M=E(\pi)$.
 \item For any left endpoint $a$ of $M$, if $a+1$ is a right endpoint of $M$,
then $(a,a+1)$ is an arc of~$M$.
 \item If $(a,b)$ and $(c,d)$ are two arcs of $M$ such that
$a<c<b<d$, then $b \neq c+1$.
\end{enumerate}
\end{lemma}
\begin{proof}
One may easily observe that the second and third property are equivalent, and
that the first property implies the second one. We now show that the second
property implies the first. Let $M$ be a matching satisfying the
second property. Construct a lattice path $P$ from $(0,n)$ to $(n,0)$,
consisting of down-steps and right-steps of unit length, where the $a$-th step
of $P$ is a down-step if and only if the point $a$ is a left endpoint of the
matching~$M$. We label the steps of $P$ from 1 to $2n$ in the order in which
they appear on~$P$. In this way, the endpoint $a$ of $M$ corresponds naturally
to the $a$-labelled step of~$P$. 

We construct a permutation $\pi$ as follows: for every arc $(a,b)$ of $M$,
let $i$ be the column containing the $b$-labelled step of $P$ (which
is a right-step) and let $j$ be the row containing the $a$-labelled step of $P$
(which is a down-step). We then define~$\pi(i)=j$. 

By construction, all elements of $\pi$ are above the path~$P$. Moreover, since
$M$ satisfies the second property, we see that whenever a down-step labelled
$a$ is directly followed by a right-step labelled $a+1$, then $(a,a+1)$ is an
arc, and the permutation $\pi$ has an LR-minimum in the column containing step
$a+1$ and row containing step~$a$. Hence $P$ is the envelope of
$\pi$, and it follows that $M=E(\pi)$.
\end{proof}
Any matching satisfying the conditions of Lemma~\ref{lem-envelope} will be
referred to as an \emph{envelope matching}. Notice that in an envelope matching,
every long arc has a short arc nested below it.

Let $R(\pi)$ denote the submatching of $E(\pi)$ formed by the long arcs
of~$E(\pi)$. The matching $R(\pi)$ no longer determines the permutation $\pi$
uniquely (e.g., both $R(12)$ and $R(213)$ consist of a single arc). We call
$R(\pi)$ the \emph{reduced envelope matching} of~$\pi$. The next lemma shows
how these concepts are related to avoidance of patterns of the
form~$1\oplus\sigma$.

\begin{lemma}\label{lem-rpi}
For any permutations $\sigma$ and $\pi$, the matching $R(\pi)$ contains
$\ma\sigma$ if and only if $\pi$ contains $1\oplus\sigma$. Moreover, suppose
that
$\{\sigma_1,\sigma_2,\dotsc,\sigma_k\}$ is a multiset of permutations such that
\[
 R(\pi)\in \MAv(\ma\sigma_1)\m\MAv(\ma\sigma_2)\m\dotsb\m\MAv(\ma\sigma_k).
\]
Then 
\[
 \pi\in\Av(1\oplus\sigma_1)\m\Av(1\oplus\sigma_2)\m
\dotsb\m\Av(1\oplus\sigma_k).
\]
In particular, if $\MAv(\ma\sigma)$ has a splitting
$\{\MAv(\ma\sigma_1),\dotsc,\MAv(\ma\sigma_k)\}$, then 
$\Av(1\oplus\sigma)$ has a splitting
$\{\Av(1\oplus\sigma_1),\dotsc,\Av(1\oplus\sigma_k)\}$.
\end{lemma}
\begin{proof}
Suppose that $R(\pi)$ contains $\ma\sigma$, and let $M\subseteq R(\pi)$ be a 
submatching of $R(\pi)$ isomorphic to~$\ma\sigma$.  From 
Observation~\ref{obs-ematch} and Lemma~\ref{lem-envelope}, we may easily deduce
that $E(\pi)$ has a short arc $\beta$ nested below all the arcs of~$M$. It
follows that $M\cup\{\beta\}$ is a copy of $\ma{1\oplus\sigma}$ in $E(\pi)$, and
hence $\pi$ contains~$1\oplus\sigma$. 

To prove the converse, suppose that $\pi$ has a subsequence
$S=s_0,s_1,\dotsc,s_n$ order-isomorphic to $1\oplus\sigma$. Let $\alpha_i$
denote the arc of $E(\pi)$ representing the element $s_i$ of~$\pi$. By
Observation~\ref{obs-ematch}, $\alpha_0$ is nested below all the arcs
$\alpha_1,\dotsc,\alpha_n$. Therefore, the arcs $\alpha_1,\dotsc,\alpha_n$ all
belong to $R(\pi)$, and no two of them are in series. Thus,
$\alpha_1,\dotsc,\alpha_n$ form a copy of $\ma\sigma$ in~$R(\pi)$.

Suppose now that $\pi$ is a permutation such that
$R(\pi)\in \MAv(\ma\sigma_1)\m\MAv(\ma\sigma_2)\m\dotsb\m\MAv(\ma\sigma_k)$.
Color the arcs of $R(\pi)$ by $k$ colors $c_1,\dotsc,c_k$ in such a way
that the arcs colored by $c_i$ avoid~$\ma\sigma_i$. We transfer this
$k$-coloring to the covered elements of $\pi$, by assigning to
each covered element the color of the corresponding arc of~$R(\pi)$ (recall
that a \emph{covered element} of a permutation is an element which is not a
LR-minimum). 

Let $\pi_i$ denote the subpermutation of $\pi$ formed by all the LR-minima of
$\pi$ together with all the covered elements whose color is~$c_i$. Note that
$R(\pi_i)$ is isomorphic to the submatching of $R(\pi)$ consisting of the arcs
of color~$c_i$. Hence $R(\pi_i)$ avoids $\ma\sigma_i$, and by first part of the
lemma, $\pi_i$ avoids~$1\oplus\sigma_i$. We deduce that
\[
\pi\in\Av(1\oplus\sigma_1)\m\Av(1\oplus\sigma_2)\m
\dotsb\m\Av(1\oplus\sigma_k).
\]
From this, the last claim of the lemma follows easily.
\end{proof}

\subsection{Proof of Theorem~\ref{thm-split}}\label{ssec-thmsplit}
We are ready to prove Theorem~\ref{thm-split}. Recall that the goal is to show
that if $\pi$ is a decomposable pattern different from $12$,
$132$ or $213$, then $\Av(\pi)$ is splittable. From 
Corollary~\ref{cor-easysplit}, we already know that we may restrict our
attention to the cases when $\pi$ has the form $1\oplus\sigma$ for an
indecomposable permutation~$\sigma$ of order at least three.

As we have seen,
a permutation $\sigma$ may be uniquely represented by an envelope
matching~$E(\sigma)$. However, we have also seen that the containment order of
permutations does not coincide with the containment order of the corresponding
envelope matchings. Our first goal will be to describe permutation containment
in terms of envelope matchings.

\begin{definition}
 Let $M$ be a matching with $n$ arcs, and let $I$ be an open interval on the
real line. The \emph{tangling of $M$ in $I$} is an operation which produces
a matching $N$ with $n+1$ arcs, defined as follows. First, we reorder the
endpoints of $M$ belonging $I$ in such a way that all the left endpoints
in $I$ appear to the left of all the right endpoints, while the relative
position of the left endpoints as well as the relative position of the right
endpoints remains the same. Next, we create a new short arc $\alpha$ 
whose endpoints belong to $I$, and which is nested below all the other arcs
that have at least one endpoint in~$I$. Let $N$
be the resulting matching. 
\end{definition}

See the right part of Figure~\ref{fig-tangle} for an example of a matching
$E(\tau)$ obtained by tangling a matching $E(\rho)$.

Note that if $M$ is an envelope matching, then any tangling of $M$ is again an
envelope matching.

\begin{figure}
\includegraphics[width=\textwidth]{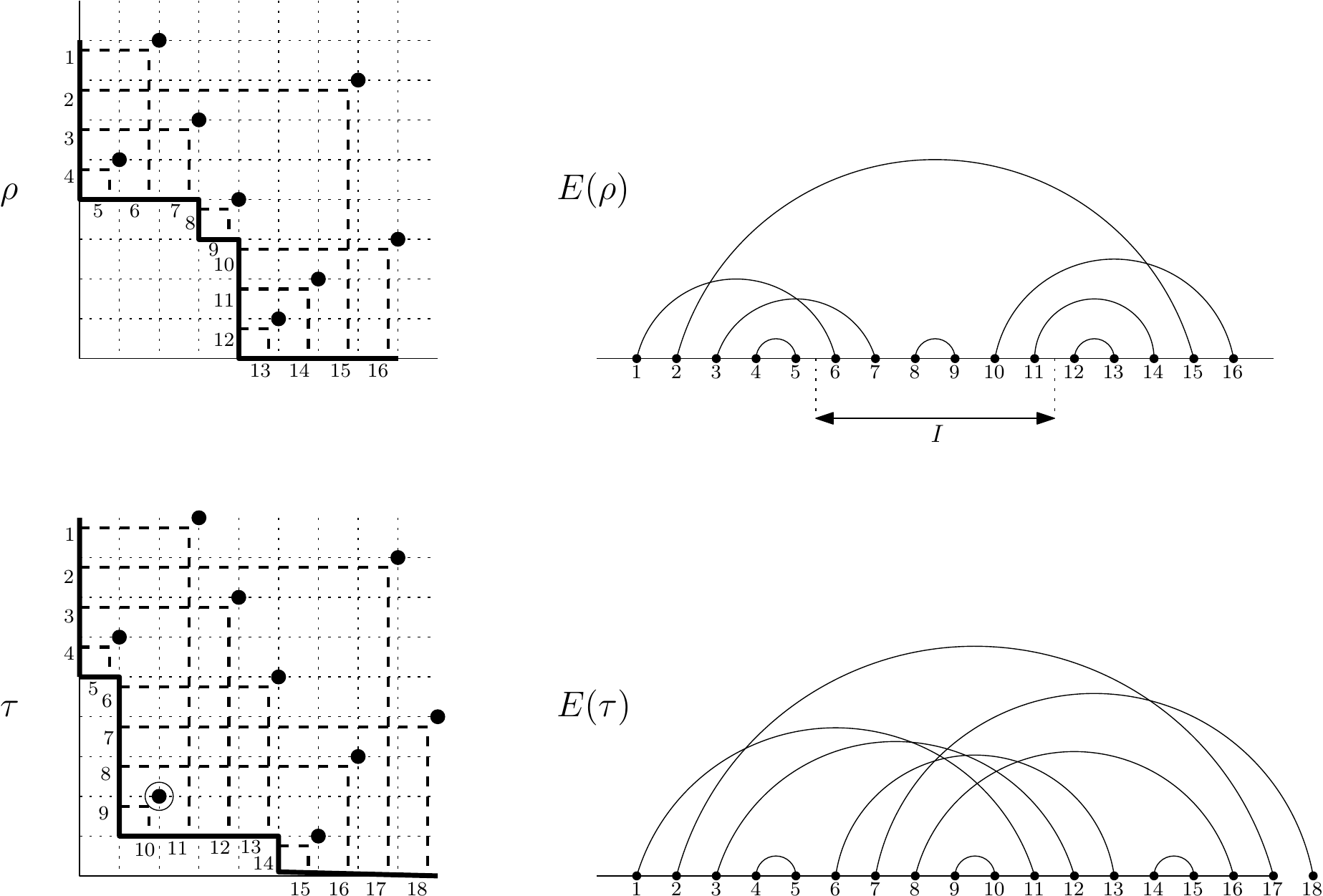}
 \caption{Illustration of Lemma~\ref{lem-tangle}: a permutation $\tau$ is
obtained by a inserting a new LR-minimum $\tau(2)$ into a permutation~$\rho$
(the new element is circled in the permutation diagram of~$\tau$). The envelope
matching $E(\tau)$ is then obtained from $E(\rho)$ by
tangling in the interval~$I$.}\label{fig-tangle}
\end{figure}

\begin{lemma}\label{lem-tangle}
Let $\tau\in S_n$ and $\rho\in S_{n-1}$
be two permutations. Suppose that $\tau$ contains $\rho$, and fix an index
$a\in[n]$ such that $\rho$ is order-isomorphic to $\tau\setminus\{\tau(a)\}$.
If $\tau(a)$ is not an LR-minimum of $\tau$, then the matching $E(\rho)$ is
contained in $E(\tau)$, and more precisely, $E(\rho)$ is obtained from
$E(\tau)$ by removing the long arc representing~$\tau(a)$. If, on the other
hand, $\tau(a)$ is an LR-minimum of $\tau$, then $E(\tau)$ can be created
from $E(\rho)$ by tangling in such a way that the short arc inserted by the
tangling operation represents~$\tau(a)$.
\end{lemma}
\begin{proof}
The case when $\tau(a)$ is not an LR-minimum follows directly from
Observation~\ref{obs-ematch}. Suppose that $\tau(a)$ is an LR-minimum of~$\tau$
(see Figure~\ref{fig-tangle}).
Let $\gamma$ be the arc of $E(\tau)$ representing~$\tau(a)$. Let
$\bar\rho=\bar\rho(1)\bar\rho(2)\dotsb\bar\rho(n-1)$ denote the sequence
$\tau\setminus\{\tau(a)\}$. That is, $\bar\rho(j)=\tau(j)$ for $j<a$ and
$\bar\rho(j)=\tau(j+1)$ otherwise. By assumption, $\bar\rho$ is order-isomorphic
to~$\rho$.

For each $i\in[n-1]$, let $\alpha_i$ be the arc representing $\rho(i)$
in $E(\rho)$, and let $\beta_i$ be the arc representing $\bar\rho(i)$ in
$E(\tau)$. Thus, $E(\tau)$ is equal to
$\{\beta_1,\dotsc,\beta_{n-1}\}\cup\{\gamma\}$.

Let us compare, for some $1\le i<j\le n-1$, the mutual position of $\alpha_i$
and $\alpha_j$ with the mutual position of $\beta_i$ and~$\beta_j$. 
By Observation~\ref{obs-ematch}, we see that
$\rend{\beta_i}<\rend{\beta_j}$ and $\rend{\alpha_i}<\rend{\alpha_j}$.
We also see that $\lend{\alpha_i}<\lend{\alpha_j}$ if and only if 
$\lend{\beta_i}<\lend{\beta_j}$. 

Observation~\ref{obs-ematch} also shows that if $\alpha_i$ and $\alpha_j$
are nested, then $\beta_i$ and $\beta_j$ are nested as well, and if $\alpha_i$
and~$\alpha_j$ are crossing, then so are $\beta_i$ and~$\beta_j$. Thus, the
only situation when the relative position of $\alpha_i$ and~$\alpha_j$ can
differ from the relative position of $\beta_i$ and~$\beta_j$ is when $\alpha_i$
and $\alpha_j$ are in series, while $\beta_i$ and~$\beta_j$ are crossing. This
happens when $\tau(a)$ covers both $\bar\rho(i)$ and $\bar\rho(j)$ in $\tau$,
but no LR-minimum of $\rho$ covers both $\rho(i)$ and~$\rho(j)$.

We partition the set $[n-1]$ into three (possibly empty) disjoint parts
$L$, $M$ and $R$ by defining $L=\{1,\dotsc,a-1\}$, $R=\{i;
\bar\rho(i)<\tau(a)\}$,
and $M=[n-1]\setminus(L\cup R)$. We will use the shorthand $\alpha_L$ to denote
the set $\{\alpha_i; i\in L\}$, and similarly for $\beta_M$, $\bar\rho_R$, etc.

Note that $\bar\rho_M$ is precisely the set of elements of $\bar\rho$ that are
covered by $\tau(a)$ in~$\tau$. Therefore, the matching $\alpha_L\cup\alpha_R$
is isomorphic to~$\beta_L\cup\beta_R$, and the arc $\gamma$ is nested below all
the arcs in~$\beta_M$. 

Let $x$ be the rightmost endpoint of an arc in $\alpha_L$. Note that an arc
$\alpha_i$ belongs to $\alpha_L$ if and only if its right endpoint is in the
interval $(-\infty,x]$. Symmetrically, let $y$ be the leftmost endpoint of an
arc in $\alpha_R$, and note that $\alpha_i$ belongs to $\alpha_R$ if and only if
its left endpoint is in $[y,\infty)$. Let $I$ be the open interval $(x,y)$, and
note that an arc $\alpha_i$ belongs to $\alpha_M$ if and only if it either has
at least one endpoint in $I$ or if $I$ is nested below it.

Combining the above facts, we see that $E(\tau)$ can be obtained by tangling
$E(\rho)$ in the interval~$I$.
\end{proof}

It follows from Lemma~\ref{lem-tangle} that if $\tau$ is a permutation that
contains a pattern $\rho$, then $E(\tau)$ may be obtained from $E(\rho)$ by
a sequence of tanglings and insertions of long arcs.

For the rest of this subsection, fix a permutation $\pi$ of the form
$1\oplus\sigma$, where $\sigma$ is an indecomposable permutation of
order at least three.

Note that the matching $E(\pi)$ is in fact isomorphic to $\ma{\pi}$, and
that $R(\pi)$ is isomorphic to~$\ma{\sigma}$. Note also that $\ma{\sigma}$ is 
connected, since $\sigma$ is indecomposable. 

The proof of Theorem~\ref{thm-split} relies on two technical lemmas. We first
state the two lemmas and prove that they imply Theorem~\ref{thm-split}, and
then prove the two lemmas themselves. 

\begin{lemma}\label{lem-split} Let $\sigma$ be an indecomposable
permutation of order at least three. Then the class $\MAv(\ma\sigma)$ is 
splittable. Furthermore, there exist two connected $\ma{\sigma}$-avoiding
matchings $M_1$ and $M_2$, such that the class $\MAv(\ma\sigma)$ admits the
splitting $\mexp{2}{\{\MAv(M_1),\MAv(M_2)\}}$.
\end{lemma}

\begin{lemma}\label{lem-perm} Suppose that $\sigma$ is an indecomposable
permutation of order at least two, $\pi$ is the permutation
$1\oplus\sigma$, and
$N$ is a $\ma{\sigma}$-avoiding matching. There is a $\pi$-avoiding
permutation $\tau\equiv\tau(N)$, such that if $\rho$ is any $\pi$-avoiding
permutation whose reduced envelope matching avoids~$N$, then $\rho$
avoids~$\tau$.
\end{lemma}

Let us show how the two lemmas imply Theorem~\ref{thm-split}. Let $M_1$ and
$M_2$ be the two matchings from Lemma~\ref{lem-split}. For each $M_i$,
Lemma~\ref{lem-perm} provides a $\pi$-avoiding
pattern~$\tau_i\equiv\tau(M_i)$. We claim that $\Av(\pi)$ has a splitting
$\mexp{2}{\{\Av(\tau_1),\Av(\tau_2)\}}$.

To see this, let $\rho$ be a $\pi$-avoiding permutation. Consider its
reduced envelope matching $R=R(\rho)$. By Lemma~\ref{lem-rpi}, $R$
avoids $\ma\sigma$, so by Lemma~\ref{lem-split}, it can be merged from four
matching $R_1$, $R_2$, $R_3$, and $R_4$, where $R_1$ and $R_2$ avoid $M_1$,
while $R_3$ and $R_4$ avoid~$M_2$.

We now define four permutations $\rho_1,\dotsc,\rho_4$, all of which are
subpermutations of $\rho$. The permutation $\rho_i$ consists of those elements
of $\rho$ which are either LR-minima of $\rho$ or which are covered and
correspond to arcs of~$R_i$. Note that $R_i$ is then precisely the reduced
envelope matching of~$\rho_i$. Since $R_1$ and $R_2$ avoid $M_1$,
Lemma~\ref{lem-perm} shows that $\rho_1$ and $\rho_2$ avoid~$\tau_1$. Similarly,
$\rho_3$ and $\rho_4$ avoid~$\tau_2$. Thus $\rho$ admits the splitting
$\mexp{2}{\{\Av(\tau_1),\Av(\tau_2)\}}$, as claimed.

\begin{proof}[Proof of Lemma~\ref{lem-split}] Let $m$ be the order of $\sigma$.
By our assumption on $\sigma$, we know that $m\ge 3$. Let $M$ be the matching
isomorphic to $\ma\sigma$ on the set of endpoints $\{1,2,\dotsc,2m\}$. Note
that $\{1,2,\dotsc,m\}$ are left endpoints of $M$, and $\{m+1,\dotsc,2m\}$ are
right endpoints.

\begin{figure}
 \hfil\includegraphics[scale=0.7]{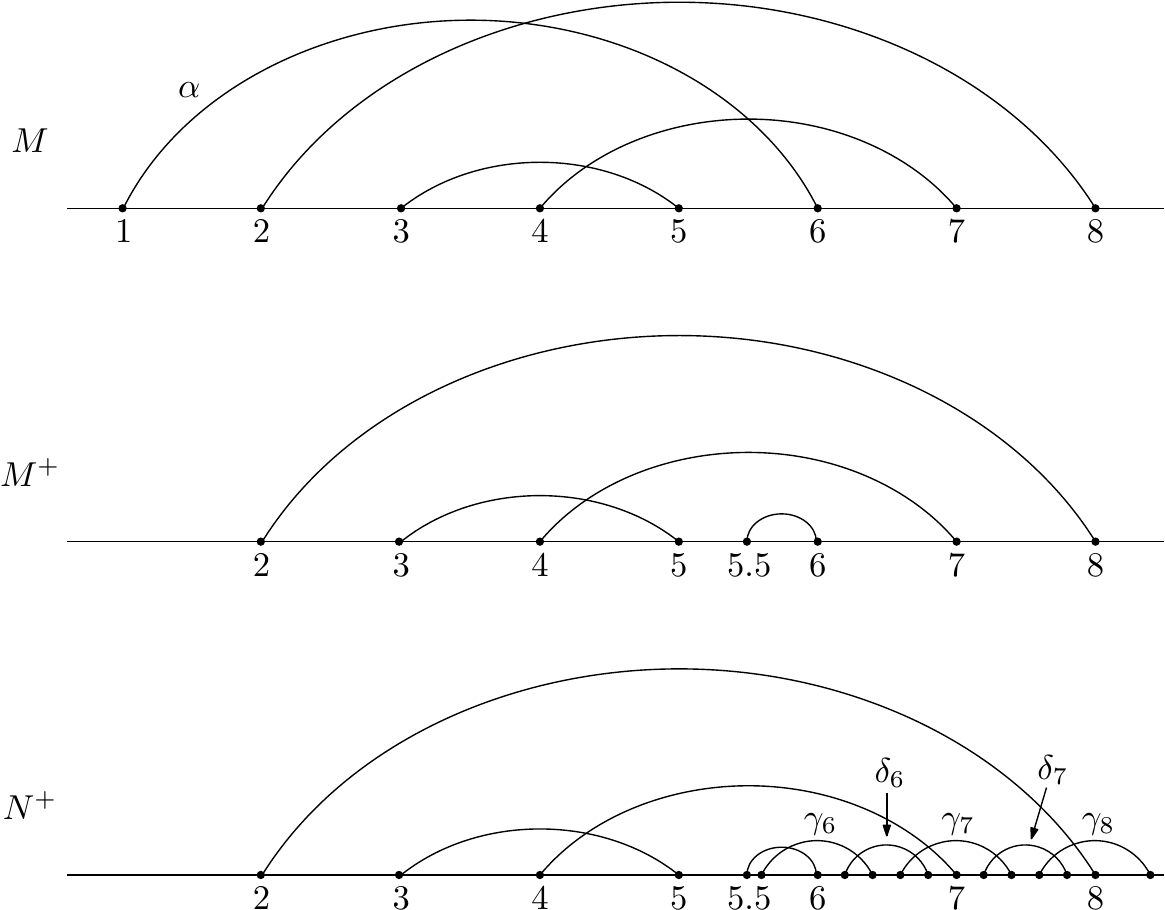}
\caption{The matching $M$ representing the permutation $\sigma=2413$ (top), and
the corresponding matchings $M^+$ (middle) and $N^+$ (bottom) used in the proof
of Lemma~\ref{lem-split}.}\label{fig-mnplus}
\end{figure}

Define two new matchings $M^+$ and $M^-$ as follows. Let $\alpha$ be the arc of
$M$ incident to the leftmost endpoint of~$M$, i.e., $\alpha$ is the arc of the
form $(1,x)$ for some $x\in\{m+1,m+2,\dotsc,2m\}$. Let $M^+$ be the matching
obtained from $M$ by removing the arc $\alpha$ and inserting a new arc with
endpoints $(x-0.5,x)$ (see Figure~\ref{fig-mnplus}). Symmetrically, define the
matching $M^-$ by considering the arc $\beta =(y,2m)\in M$ incident to the
rightmost endpoint of $M$, and replacing $\beta$ with an arc $(y,y+0.5)$.

We now show that it is possible to insert new arcs to $M^+$ in order to obtain a
connected matching $N^+$ that avoids~$M$. If $M$ has exactly three arcs, we
easily check by hand that this is possible, so let us assume that $m\ge 4$.

Let $\gamma_i$ denote the arc $(i-0.4,i+0.4)$ and let $\delta_j$ denote the
arc $(j+0.2,j+0.8)$. Define sets of arcs $\Gamma=\{\gamma_i;\; x\le i\le 2m\}$
and $\Delta=\{\delta_j;\; x\le j < 2m\}$, and consider the matching
$N^+=M^+\cup\Gamma\cup\Delta$. We claim that $N^+$ is a connected $M$-avoiding
matching.

To see that $N^+$ is connected, notice first that since $M$ is connected, every
connected component of $M^+$ has at least one arc incident to one of the
endpoints $\{x,\dotsc,2m\}$. Since any such arc is crossed by an arc of
$\Gamma$, and since $\Gamma\cup\Delta$ induce a connected matching, we see that
$N^+$ is connected.

Let us argue that $N^+$ avoids $M$. Suppose for contradiction that $N^+$ has a
submatching $\overline{M}\subseteq N^+$ isomorphic to~$M$. Note that
$\overline{M}$ is a permutation matching, i.e., it has no two arcs in series.
In particular, $\overline{M}$ contains at most one arc from $\Gamma$ and at most
one arc from $\Delta$. Since $\overline{M}$ is connected and has at most one arc
of $\Gamma$, it may not contain arcs from two distinct components of~$M^+$. On
the other hand, $\overline{M}$ must contain at least one arc
$\gamma_i\in\Gamma$, otherwise it would be a subset of a single component of
$M^+$, which is impossible. Therefore, $\overline{M}$ has no arc whose right
endpoint is to the left of $x$, since any such arc is in series with all the
arcs of~$\Gamma$, and $\overline{M}$ does not contain the arc $(x-0.5,x)$
since this arc does not belong to any connected permutation submatching of
$N^+$ having more than two arcs.

Let $\cro(M)$ denote the number of crossings in the matching $M$, i.e., the
number of edges in the intersection graph of~$M$. We see that
$\cro(M^+)=\cro(M)-(2m-x)$. Moreover, $\cro(M)=\cro(\overline{M})\le
\cro(M^+)+2$, since $\overline{M}$ is obtained by adding at most one arc of
$\gamma$ and a most one arc of $\Delta$ to a submatching of~$M^+$. This shows
that $x\ge 2m-2$. We see that $x\neq 2m$, because $M$ is connected. Also,
$x\neq 2m-1$, because $\overline{M}$ has at least four arcs. This leaves the
possibility $x=2m-2$ and $\overline{M}$ contains the two arcs of $M$ with
endpoints $2m-1$ and $2m$, together with the two arcs $\gamma_{2m-1}$ and
$\delta_{2m-2}$. Since $M$ is connected, this only leaves the possibility when
$M$ is the matching $\{ (1,6),(2,8),(3,5),(4,7)\}$, but in this case we
easily check that $N^+$ has no copy of~$M$.

We conclude that $N^+$ is a connected matching containing $M^+$ but not~$M$. By
a symmetric argument, we obtain a connected matching $N^-$ containing $M^-$ but
not~$M$. We now show that Lemma~\ref{lem-split} holds with $M_1=N^+$ and
$M_2=N^-$, i.e., we show that the class of matchings $\MAv(M)=\MAv(\ma\sigma)$
admits the splitting $\cS=\mexp{2}{\{\MAv(N^+),\MAv(N^-)\}}$. 

Let $R$ be any $\ma\sigma$-avoiding matching. Note that a disjoint union of 
$N^+$-avoiding matchings is again $N^+$-avoiding, and the same is true for 
$N^-$, since both $N^+$ and $N^-$ are connected. In particular, if each 
connected component of $R$ admits the splitting $\cS$, then so
does $R$. We may therefore assume that $R$ is connected.

Recall that $\uint{R}$ is the intersection graph of~$R$. Let $\alpha_0\in R$ be
the arc that contains the leftmost endpoint of~$R$. We partition the arcs of $R$
into sets $L_0,L_1,L_2,\dotsc$, where an arc $\beta$ belongs to $L_i$ if and
only if the shortest path between $\alpha_0$ and $\beta$ in $\uint{R}$ has
length~$i$. We refer to the elements of $L_i$ as \emph{arcs of level $i$}. In
particular, $\alpha_0$ is the only arc of level~0. Since $R$ is connected, the
level of each arc is well defined. 

Clearly, an arc of level $i$ may only cross arcs of level $i-1$, $i$ or
$i+1$. In particular, if each level $L_i$ admits the splitting
$\cS'=\{\MAv(N^+),\MAv(N^-)\}$, then the union $L_0\cup L_2\cup L_4\cup\dotsb$
of all the even levels admits this splitting as well, and the same is true for
the union of the odd layers. It then follows that the matching $R$ admits the
splitting~$\mexp{2}{\cS'}=\cS$.

It is thus enough to show, for each fixed $i$, that $L_i$ admits the splitting
$\cS'$. Since $L_0$ only contains the arc $\alpha_0$, we focus on the remaining
levels, and fix $i\ge 1$ arbitrarily. Note that each arc of level $i$ crosses at
least one arc of level $i-1$. For an arc $\beta\in L_i$, let $\nu(\beta)$ be an
arbitrary arc of level $i-1$ crossed by~$\beta$. We now partition $L_i$ into two
sets $L_i^+$ and $L_i^-$ as follows:
\begin{align*}
 L_i^+&=\{\beta\in L_i\colon \nu(\beta)\ \text{crosses}\ \beta\ \text{from the
left}\}\text{, and}\\
 L_i^-&=\{\beta\in L_i\colon \nu(\beta)\ \text{crosses}\ \beta\ \text{from the
right}\}. 
\end{align*}
Our goal is to show that $L_i^+$ avoids $N^+$ and $L_i^-$ avoids $N^-$.

Recall that a block of a matching is a maximal $\uplus$-indecomposable
submatching. The rest of the proof of Lemma~\ref{lem-split} is based
on two claims.

\begin{claim}\label{cla-blockx}
Let $B$ be a block of the matching $L_i^+$, let $x$ be the leftmost endpoint
of $B$, and let $\beta$ be and arc of~$B$. Then the left endpoint of
$\nu(\beta)$ is to the left of~$x$. Similarly, if $B'$ is a block of $L_i^-$,
$x'$ is the rightmost endpoint of $B'$, and $\beta'$ is an arc of $B'$, then
the right endpoint of $\nu(\beta')$ is to the right of~$x'$.
\end{claim}
\begin{proof}[Proof of Claim~\ref{cla-blockx}]
We prove the claim for $L_i^+$, the case of $L_i^-$ is analogous. If $i=1$,
then $\nu(\beta)$ is the arc $\alpha_0$, and the claim holds. Suppose now that
$i>1$. Let $y$ be the rightmost endpoint of~$B$. Note that any arc of $R$ that
has exactly one endpoint in the interval $[x,y]$ must cross at least one arc of
$B$, since $B$ is $\uplus$-indecomposable. In particular, any arc having
exactly one endpoint in $[x,y]$ must have level at least~$i-1$.

Let $P=(\alpha_0,\alpha_1,\dotsc,\alpha_{i-1},\beta)$ be a shortest path in
$\uint{R}$ between $\alpha_0$ and $\beta$, chosen in such a way that
$\alpha_{i-1}=\nu(\beta)$. Since $\alpha_0$ contains the leftmost endpoint of
$R$, both endpoints of $\alpha_0$ are outside of $[x,y]$, and the same is true
for every $\alpha_j$ with $j<i-1$. Consequently, the arc $\nu(\beta)$ must have
at least one endpoint outside of $[x,y]$, and since $\nu(\beta)$ crosses
$\beta$ from the left, the left endpoint of $\nu(\beta)$ is to the left of~$x$,
and the claim is proved.
\end{proof}

\begin{claim}\label{cla-mplusx} 
 Every block of the matching $L_i^+$ avoids $M^+$, and every block of $L_i^-$
avoids~$M^-$.
\end{claim}
\begin{proof}[Proof of Claim~\ref{cla-mplusx}] Let us prove the claim for
$L_i^+$, the other part is symmetric. Let $B$ be a block of $L_i^+$, and suppose
that it contains a matching $\overline{M^+}$ isomorphic to~$M^+$. Let $\zeta$
be the arc of $\overline{M^+}$ that corresponds to the arc $(x-0.5,x)\in
M^+$ via the isomorphism between $\overline{M^+}$ and $M^+$. By
Claim~\ref{cla-blockx}, the left endpoint of $\nu(\zeta)$ is to the left of any
endpoint of $B$, and in particular, it is to the left of any endpoint of
$\overline{M^+}$. But this means that the matching
$\overline{M^+}\cup\{\nu(\zeta)\}\setminus\{\zeta\}$ is isomorphic to $M$,
contradicting the assumption that $R$ is $M$-avoiding.
\end{proof}

Since every block of $L_i^+$ avoids $M^+$, it also avoids $N^+$. This means
that $L_i^+$ avoids $N^+$ as well. Similarly, $L_i^-$ avoids~$N^-$, showing
that $L_i$ admits the splitting~$\cS'$. This means that $R$ admits the
splitting~$\cS$, and Lemma~\ref{lem-split} is proved.
\end{proof}

\begin{proof}[Proof of Lemma~\ref{lem-perm}]
Consider a $\ma\sigma$-avoiding matching~$N$. Suppose that its point set is the
set $[2n]$. Our goal is to construct a $\pi$-avoiding permutation $\tau$ whose
envelope matching $T$ has the property that any sequence of tanglings and long
arc insertions performed on $T$ will create a matching whose long arcs contain
a copy of $N$ or of~$\ma\sigma$. It will then be easy to verify that such a
permutation $\tau$ has the properties stated in the lemma.

Let $m\ge 2$ be the order of~$\sigma$. Let $M$ be the matching isomorphic to
$\ma\sigma$ on the point set $[2m]$. By assumption, $\sigma$ is
indecomposable, so $M$ is connected. Let $i\ge m+1$ be the endpoint
connected to $m$ by an arc of $M$, and let $j\le m$ be the endpoint connected
to $m+1$. Note that $m$ and $m+1$ are not connected by an arc of $M$, since $M$
is connected. Let $M'$ be the matching obtained from $M$ by replacing the two
arcs $(m,i)$ and $(j,m+1)$ by the arcs $(m+1,i)$ and $(j,m)$.

We now construct the required matching $T$, by adding new arcs to~$N$.
For every two consecutive endpoints $a$ and $a+1$ of $N$ such that $a$ is a
right endpoint and $a+1$ is a left endpoint, we insert into $N$ an isomorphic copy
of $M'$ whose endpoints all belong to the interval $(a,a+1)$. Let $N'$ be the
resulting matching. Note that $N'$ is $\ma\sigma$-avoiding, since
$\ma\sigma$ is connected and each connected component of $N'$ is a subset of $N$
or a subset of an isomorphic copy of~$M'$. 

Next, we create the envelope matching $T$, by inserting a new short arc between
any pair of consecutive endpoints $b<c$ of $N'$ such that $b$ is a left endpoint
and $c$ is a right endpoint. This means that the long arcs of $T$ are precisely
the arcs of~$N'$. Let $\tau$ be the permutation whose envelope matching
is~$T$, and therefore its reduced envelope matching is~$N'$. 

We claim that $\tau$ has the required properties. By Lemma~\ref{lem-rpi}, $\tau$
avoids~$\pi$. It remains to show that if $\rho$ is a $\pi$-avoiding permutation
containing $\tau$, then the reduced envelope matching of $\rho$ contains~$N$.
Suppose that $\rho$ is a $\pi$-avoiding permutation which contains~$\tau$,
and suppose for contradiction that the matching $R=R(\rho)$ avoids~$N$.

By Lemma~\ref{lem-tangle}, the envelope matching of $\rho$ can be created from
$T$ by a sequence of tanglings and insertions of new long arcs. Since the long
arcs of $T$ contain the arcs of $N$ as a subset while $R(\rho)$ avoids $N$, we
see that one of the tanglings used to create $E(\rho)$ from $T$ must have
altered the relative order of the endpoints of~$N$. This means that the tangling
was performed on an interval $I$ that contained a right endpoint $a$ of $N$
situated to the left of a left endpoint $a+1$ of~$N$. By construction, $N'$
contains a copy of $M'$ in the interval $(a,a+1)$. The tangling on $I$ will
change this copy of $M'$ into a copy of $M$, and insert a new short arc nested
below all the arcs of $M$, thus creating a copy of~$\ma\pi$. Since any further
tangling or any further insertion of long arcs cannot affect the relative order
of endpoints in this copy of $\ma\pi$, we see that $E(\rho)$ contains $\ma\pi$,
and therefore $\rho$ contains $\pi$, a contradiction. 

This completes the proof of Lemma~\ref{lem-perm} and of
Theorem~\ref{thm-split}.
\end{proof}

\subsection{More natural splittings}\label{ssec-expl}

The splittings obtained by applying Theorem~\ref{thm-split} involve patterns
which are large and have complicated structure.
It is unlikely that such splittings could be directly used for enumeration
purposes. In special cases, it is however possible to apply a different
argument which yields more natural and more explicit splittings. One such
result, stated in the next theorem, is the focus of this subsection.

\begin{theorem}\label{thm-spsp}
Let $\pi$ be an indecomposable permutation of order~$n$. If the
class $\Av(\pi)$ admits a splitting
\[
 \{\Av(\pi_1),\Av(\pi_2),\dotsc,\Av(\pi_k)\}
\]
for a multiset $\{\pi_1,\dotsc,\pi_k\}$ of indecomposable
permutations, then $\Av(1\oplus\pi)$ admits the splitting
\[
 \mexp{K}{\{\Av(1\oplus\pi_1),\Av(1\oplus\pi_2),\dotsc,\Av(1\oplus\pi_k)\}},
\]
for some $K\le 16^n$.
\end{theorem}

Before we prove this theorem, we show that the assumption that
$\pi_1,\dotsc,\pi_k$ are indecomposable does not restrict
the theorem's applicability.

\begin{lemma}\label{lem-indeco}
 Let $\pi$ be an indecomposable permutation. Suppose that
$\Av(\pi)$ has the splitting
\begin{equation}
 \cS=\{\Av(\pi_1),\Av(\pi_2),\dotsc,\Av(\pi_k)\},
\label{eq-indeco1}
\end{equation}
for some permutations $\pi_1,\dotsc,\pi_k$. Suppose furthermore that
$\pi_1$ can be written as $\pi_1=\pi_1'\oplus\pi_1''$. Then
$\Av(\pi)$ also admits the splitting 
\begin{equation}
 \cS'=\{\Av(\pi_1'),\Av(\pi_1''),\Av(\pi_2),\dotsc,
\Av(\pi_k)\}.\label{eq-indeco2}
\end{equation}
\end{lemma}
\begin{proof}
To show that $\cS'$ is a splitting of $\Av(\pi)$, we will prove the following
stronger statement:
\[
 \Av(\pi)\!\subseteq\!\left(\Av(\pi_1')\m\Av(\pi_2)\m\dotsb\m
\Av(\pi_k)
 \right)\cup\left(\Av(\pi_1'')\m\Av(\pi_2)\m\dotsb\m\Av(\pi_k)\right).
\]

Let $\rho$ be a $\pi$-avoiding permutation, and let $\tau$ be the permutation
$\rho\oplus\rho$. Since $\pi$ is indecomposable, $\tau$
avoids~$\pi$, and therefore it can be colored by $k$ colors $c_1,\dotsc,c_k$ so
that the elements of color $c_i$ avoid~$\pi_i$. The permutation $\tau$ is a
direct sum of two copies of $\rho$; we refer to them as the bottom-left copy and
the top-right copy. We consider the restrictions of the coloring of $\tau$ to
the two copies. If the bottom-left copy contains $\pi_1'$ in color $c_1$ and the
top-right copy contains $\pi_1''$ in color $c_1$, then $\tau$ contains $\pi_1$
in color $c_1$, which is impossible. 

Suppose that the bottom-left copy has no $\pi_1'$ in color~$c_1$. Then the
coloring of the bottom-left copy of $\rho$ demonstrates that
$\rho\in\Av(\pi_1')\m\Av(\pi_2)\m\dotsb\m\Av(\pi_k)$, as claimed. Similarly, if
the top-right copy of $\rho$ avoids $\pi_1''$ in color $c_1$, we see that 
$\rho\in\Av(\pi_1'')\m\Av(\pi_2)\m\dotsb\m\Av(\pi_k)$. This proves the lemma.
\end{proof}
\begin{corollary}
If $\pi$ is indecomposable and the class $\Av(\pi)$ is splittable,
then $\Av(\pi)$ has a splitting of the form $\{\Av(\pi_1),\dotsc,\Av(\pi_k)\}$,
where each $\pi_i$ is an indecomposable $\pi$-avoiding permutation.
\end{corollary}

The rest of this subsection is devoted to the proof of Theorem~\ref{thm-spsp}.
As in the proof of Theorem~\ref{thm-split}, we use the formalism of ordered
matchings. Note that by Observation~\ref{obs-match}, a class $\Av(\pi)$ has a
splitting $\{\Av(\pi_1),\dotsc,\Av(\pi_k)\}$ if and only if the class of
permutation matchings $\ma{\Av(\pi)}$ has a
splitting $\{\ma{\Av(\pi_1)},\dotsc,\ma{\Av(\pi_k)}\}$. On the other hand, by
Lemma~\ref{lem-rpi}, a splitting of $\Av(1\oplus\pi)$ into parts of the form
$\Av(1\oplus\pi_i)$ can be obtained from a splitting of $\MAv(\ma\pi)$ into
parts of the form $\MAv(\ma{\pi_i})$. Therefore, what we need is an argument
showing how a splitting of the class $\ma{\Av(\pi)}$ can be transformed into a
splitting of the larger class $\MAv(\ma\pi)$. For this, we state and prove the
next theorem.

\begin{theorem}\label{thm-match}
Let $\pi$ be a permutation of order $n$. Suppose that the class
$\ma{\Av(\pi)}$ has a splitting
\[
\{\ma{\Av(\pi_1)},\ma{\Av(\pi_2)},\dotsc,\ma{\Av(\pi_k)}\},
\]
where each $\pi_i$ is a nonempty indecomposable permutation.
Then there is a constant $K\le 16^{n}$ such that the class of matchings
$\MAv(\ma\pi)$ has a splitting
\[
\mexp{K}{\{\MAv(\ma\pi_1),\MAv(\ma\pi_2),\dotsc,
\MAv(\ma\pi_k)\}}.
\]
\end{theorem}

The theorem essentially says that if we can split a $\ma\pi$-avoiding
permutation matching into matchings that avoid some other permutation patterns,
then we can have an analogous splitting also for non-permutation $\ma\pi$-avoiding
matchings. The splitting of non-permutation matchings may need a $K$-fold
increase in the number of parts used. 

The bound $K\le 16^{n}$ can be slightly
improved by a more careful proof, but the particular approach we will use in the
proof seems to always produce a bound of the form $c^n$ for
some~$c$. Getting a subexponential upper bound for $K$ would be a major
achievement, as this would imply, as a special case, a subexponential bound on
the chromatic number of circle graphs avoiding a clique of size~$n$, as we shall
see in Subsection~\ref{ssec-circ}.

Recall that for a matching $M$, an arc $\alpha\in M$ is long
if at least one endpoint of $M$ is nested
below~$\alpha$; otherwise it is short. Let $\lo(M)$
be the number of long arcs of~$M$. Define the \emph{weight} of $M$, denoted by
$\wei(M)$, as $\wei(M)=|M|+\lo(M)$. To prove Theorem~\ref{thm-match}, we will in
fact prove the following more general result.

\begin{lemma}\label{lem-match}
Let $\pi$ be a permutation of order $n$. Suppose that the class
$\ma{\Av(\pi)}$ has a splitting
\[
\{\ma{\Av(\pi_1)},\ma{\Av(\pi_2)},\dotsc,\ma{\Av(\pi_k)}\},
\]
where each $\pi_i$ is a nonempty indecomposable permutation.
Let $M$ be a matching. There is a constant $K\le 4^{\wei(M)}$ such that
$\MAv(\{\ma\pi,M\})$ has a splitting
\[
 \mexp{K}{\{\MAv(\ma\pi_1),\MAv(\ma\pi_2),\dotsc, \MAv(\ma\pi_k)\}}.
\]
\end{lemma} 
To deduce Theorem~\ref{thm-match} from Lemma~\ref{lem-match}, we simply
choose $M=\ma\pi$ and observe that $\wei(M)\le 2|M|$. 

\begin{proof}[Proof of Lemma~\ref{lem-match}]
We proceed by induction on $\wei(M)$. Let
$\cS$ denote the multiset $\{\MAv(\ma\pi_1),\MAv(\ma\pi_2),\dotsc,
\MAv(\ma\pi_k)\}$, and let $K(M)$ denote the smallest value of $K$ for which
$\mexp{K}{\cS}$ is a splitting of $\MAv(\{\ma\pi,M\})$. We want to prove
that
this value is finite for each~$M$, and in particular, that $K(M)\le
4^{\wei(M)}$. We will assume, without loss of generality, that the endpoints of
$M$ form the set~$[2m]$, where $m$ is the size of $M$.

If $\wei(M)=1$, then $M$ consists of a single arc, so the class
$\MAv(\{\ma\pi,M\})$ only contains the empty matching, and we have $K(M)=1$.
Suppose now that $\wei(M)>1$. We will distinguish several cases, based on the
structure of~$M$. 

Suppose first that $M$ has the form $M_1\uplus M_2$ for some nonempty matchings
$M_1$ and~$M_2$. Note that both $M_1$ and $M_2$ have smaller weight than~$M$.

Choose a matching $N\in\MAv(\{\ma\pi,M\})$ with $q$ arcs. Suppose that $N$
has the point set $[2q]$. Let $N[<\!i]$ be the submatching of $N$
consisting of all those arcs whose both endpoints are smaller than $i$,
and symmetrically, $N[>\!i]$ consists of those arcs whose both endpoints are
larger than~$i$. $N[\le\!i]$ and $N[\ge\! i]$ are defined analogously.

Fix the smallest value of $i$ such that $N[\le\! i]$ contains a copy of~$M_1$.
If no such value exists, it means that $N$ belongs to $\MAv(\{\ma\pi,M_1\})$,
and admits the splitting $\mexp{K(M_1)}{\cS}$ by induction. Suppose then
that $i$ has been fixed. Then $N[>\!i]$ avoids $M_2$, otherwise $N$ would
contain~$M$. We partition $N$ into three disjoint matchings $N[<\!i]$,
$N[>\!i]$, and $N_0=N\setminus (N[<\!i]\cup N[>\!i])$. The first of these
matchings avoids $M_1$, the second avoids $M_2$, and the third one is a
permutation matching. It follows that these matchings admit the splittings
$\mexp{K(M_1)}{\cS}$, $\mexp{K(M_2)}{\cS}$, and $\cS$, respectively, and
hence $N$ admits the splitting
$\mexp{(K(M_1)+K(M_2)+1)}{\cS}$. We conclude that $K(M)\le 2
(4^{\wei(M)-1})+1\le 4^{\wei(M)}$.

Suppose from now on that $M$ is $\uplus$-indecomposable. This part of our
argument is analogous to the proof of Lemma~\ref{lem-split}. Fix a matching
$N\in\MAv(\{\ma\pi,M\})$. Recall that the permutations $\pi_1, \dotsc,
\pi_k$ are indecomposable by assumption, and therefore the matchings
$\ma\pi_i$ are all connected. Therefore, if for some $K$ each connected
component of $N$ admits a splitting $\mexp{K}{\cS}$, then $N$ admits this
splitting as well. We will therefore assume from now on that the matching $N$ is
connected. Recall that $\uint{N}$ is the intersection graph of~$N$.

Let $\alpha_0\in N$ be the arc that contains the leftmost endpoint of~$N$. We
partition the arcs of $N$ into sets $L_0,L_1,L_2,\dotsc$, where an arc $\beta$
belongs to $L_i$ if and only if the shortest path between $\alpha_0$ and $\beta$
in $\uint{N}$ has length~$i$. We refer to the elements of $L_i$ as \emph{arcs of
level $i$}.

Since $\ma\pi_i$ is connected for each $i$, we see that if each level $L_i$
admits a splitting $\mexp{K}{\cS}$, then the union $L_0\cup L_2\cup
L_4\cup\dotsb$
of all the even levels admits this splitting as well, and the same is true for
the union of the odd layers. It then follows that the matching $N$ admits the
splitting~$\mexp{2K}{\cS}$.

It is thus sufficient to find, for each fixed $i$, a splitting of $L_i$ of the
form $\mexp{K}{\cS}$. Since $L_0$ only contains the arc $\alpha_0$, we
focus on
$i\ge 1$. As in the proof of Lemma~\ref{lem-split}, for an arc $\beta\in L_i$,
we let $\nu(\beta)$ be an arbitrary arc of level $i-1$ crossed by $\beta$. We
again partition $L_i$ into two sets $L_i^+$ and $L_i^-$, with
\begin{align*}
 L_i^+&=\{\beta\in L_i\colon \nu(\beta)\ \text{crosses}\ \beta\ \text{from the
left}\}\text{, and}\\
 L_i^-&=\{\beta\in L_i\colon \nu(\beta)\ \text{crosses}\ \beta\ \text{from the
right}\}. 
\end{align*}

We now define two matchings $M^+$ and $M^-$, both of strictly smaller weight
than $M$, and show that the blocks of $L_i^+$ and $L_i^-$ avoid
$M^+$ and $M^-$, respectively.

Let $\gamma$ be the arc of $M$ that contains the leftmost endpoint of~$M$, and
let $x$ be its other endpoint; in other words, we have $\gamma=\{1,x\}$.
Note that $\gamma$ is long, otherwise $M$ would be $\uplus$-decomposable.
Let $M^+$ be the matching obtained from $M$ by replacing the arc $\gamma$ with
the arc $\{x-0.5,x\}$. Since we replaced a long arc $\gamma\in M$ by a short
arc $\{x-0.5,x\}\in M^+$, we see that $\wei(M^+)=\wei(M)-1$.

We define the matching $M^-$ symmetrically. Let $\gamma'$ be the arc of $M$
containing the rightmost endpoint of~$M$, and let $y$ be the left endpoint of
$\gamma'$. We let $M^-$ denote the matching obtained by replacing
$\gamma'$ by the arc~$\{y,y+0.5\}$.

Repeating the arguments of Claim~\ref{cla-blockx} and
Claim~\ref{cla-mplusx} in the proof of Lemma~\ref{lem-split}, 
we again see that every block of $L_i^+$ avoids $M^+$ and every block of
$L_i^-$ avoids~$M^-$. By induction, every block of $L_i^+$ admits the splitting
$\mexp{K(M^+)}{\cS}$, and consequently, the whole matching $L_i^+$ admits
such
splitting as well. Similarly, $L_i^-$ admits the splitting
$\mexp{K(M^-)}{\cS}$.

Consequently, $L_i$ admits the splitting $\mexp{(K(M^+)+K(M^-))}{\cS}$.
This
implies that both $L_0\cup L_2\cup L_4\cup\dotsb$ and $L_1\cup L_3\cup
L_5\cup\dotsb$ admit the splitting $\mexp{(K(M^+)+K(M^-))}{\cS}$. We
conclude
that $N$ has a splitting $\mexp{K}{\cS}$, with $K\le 2(K(M^+)+K(M^-))\le
4^{\wei(M)}$, as claimed. 

This completes the proof of Lemma~\ref{lem-match} and
therefore also of Theorem~\ref{thm-match} and Theorem~\ref{thm-spsp}.
\end{proof}

\subsection{Splittings and $\chi$-boundedness of circle graphs}\label{ssec-circ}
Let $\J_n$ denote the decreasing permutation $n(n-1)\dotsb 1$. It is well known
that any $\J_n$-avoiding permutation can be merged from at most $n-1$ increasing
sequences; in other words, $\Av(\J_n)$ has a splitting
$\mexp{(n-1)}{\{\Av(21)\}}$.
By Theorem~\ref{thm-spsp}, we conclude that
$\Av(1\oplus\J_n)$ has a splitting $\mexp{K(n-1)}{\{\Av(132)\}}$ for $K\le
16^n$.

Let $f(n)$ denote the smallest integer such that $\Av(1\oplus\J_n)$ admits the
splitting $\mexp{f(n)}{\{\Av(132)\}}$. Recall that a \emph{circle graph}
is a
graph that can be obtained as the intersection graph of a matching. A
\emph{proper coloring} of a graph $G$ is the coloring of the vertices of $G$ in
such a way that no two adjacent vertices have the same color. The
\emph{chromatic number} of~$G$ is the smallest number of colors needed for a
proper coloring of~$G$. A graph is \emph{$K_n$-free} if it has no complete
subgraph on $n$ vertices.

\begin{proposition}\label{pro-circ}
Let $k$ and $n$ be integers. Then the class $\Av(1\oplus\J_n)$
admits the splitting $\mexp{k}{\{\Av(132)\}}$ if and
only if every $K_n$-free circle graph can be properly colored by
$k$ colors. In particular, $f(n)$ is equal to the largest chromatic number of
$K_n$-free circle graphs.
\end{proposition}
\begin{proof}
Let $M$ be a matching, and let $G$ be its intersection graph. Note that $G$ is
$K_n$-free if and only if $M$ has no $n$ pairwise crossing arcs, i.e., $M$ is 
$\ma{\J_n}$-avoiding. Note also that $G$ can be properly $k$-colored if and
only if $M$ admits the splitting $\mexp{k}{\{\MAv(\ma{\J_2})\}}$. Thus, if
every
$K_n$-free circle graph can be properly $k$-colored, then $\MAv(\ma{\J_n})$
admits the splitting $\mexp{k}{\{\MAv(\ma{\J_2})\}}$, and by
Lemma~\ref{lem-rpi},
$\Av(1\oplus\J_n)$ admits the splitting $\mexp{k}{\{\Av(132)\}}$.

Conversely, assume that $\Av(1\oplus\J_n)$ admits the splitting
$\mexp{k}{\{\Av(132)\}}$. We will show that every $K_n$-free circle graph
can be
properly $k$-colored. It is enough to show this for the smallest possible value
of $k$, i.e., for $k=f(n)$. In particular, we may assume that there is a
permutation $\rho\in\Av(1\oplus\J_n)$ that does not admit the splitting
$\mexp{(k-1)}{\{\Av(132)\}}$.

Let $G$ be a $K_n$-free circle graph, and suppose that $G$ is the intersection
graph of a matching~$M$. Let $\sigma$ be a permutation whose reduced envelope
matching is~$M$. By Lemma~\ref{lem-rpi}, $\sigma$ is $(1\oplus \J_n)$-avoiding.
Let $\pi$ be the permutation obtained by simultaneously inflating each
LR-minimum of $\sigma$ by a copy of~$\rho$. Note that $\pi$ is also
$(1\oplus\J_n)$-avoiding. Therefore $\pi$ has a coloring $c$ by $k$ colors with
no monochromatic copy of 132. Every vertex $v$ of $G$ corresponds to an arc of
$M$, which will be denoted by~$\alpha[v]$. In turn, the arc $\alpha[v]$
corresponds to a covered element $\sigma[v]$ of $\sigma$, and this corresponds
to an element $\pi[v]$ of~$\pi$. Let us define a coloring $c'$ of the vertices
of~$G$ by putting $c'(v)=c(\pi[v])$.

We claim that $c'$ is a proper coloring. To see this, pick two adjacent
vertices $u$ and $v$ of~$G$. Then $\alpha[u]$ and $\alpha[v]$ are two crossing
arcs in $M$, and $\sigma[u]$ and $\sigma[v]$ are two covered elements of
$\sigma$ with the property that $\sigma$ has an LR-minimum $\sigma(i)$ that
covers both $\sigma[u]$ and~$\sigma[v]$. 

Note that in every occurrence of $\rho$ in $\pi$, the coloring $c$ must use all
$k$ colors, because $\rho$ does not admit the splitting
$\mexp{(k-1)}{\{\Av(132)\}}$. In particular, the copy of $\rho$ formed by
inflating $\sigma(i)$ must contain all $k$ colors. Consequently, $\pi[u]$ and
$\pi[v]$ must have distinct colors in order to avoid a monochromatic copy of
132. Thus, $c'(u)\neq c'(v)$, showing that $c'$ is a proper coloring.
\end{proof}

The problem of estimating the value of $f(n)$, i.e., the largest chromatic
number of $K_n$-free circle graphs, has been studied by graph theorists since
the 1980s. Gyárfás~\cite{gyar,gyarcor} was the first to prove that this
chromatic number is bounded, and showed that $f(n)\le n^2 2^n(2^n-2)$. This has
been later improved by Kostochka and Kratochvíl~\cite{koskra}, who proved that
$f(n)\le 50\cdot2^n-O(n)$. They in fact showed that this bound is also
applicable to the so-called polygon-circle graphs, which are intersection graphs
of sets of polygons inscribed into a common circle, and are easily seen to be a
generalization of circle graphs. Currently, the best known upper bound for
$f(n)$ is due to Černý~\cite{Cerny}, who proved that $f(n)\le 21\cdot
2^n-O(n)$. All these bounds are still far away from the best known general
lower bound $f(n)\ge\Omega(n\log n)$, proven by Kostochka~\cite{koslow}.

For specific values of $n$, better estimates are known. For instance, $f(3)=5$,
as shown by Kostochka~\cite{kostri}, who proved that $f(3)\le 5$, and
Ageev~\cite{ageev}, who constructed an example of $K_3$-free 5-chromatic circle
graph on 220 vertices. Recently, Nenashev~\cite{nena} has also shown that
$f(4)\le 30$. 

In general, the problem of estimating $f(n)$ appears rather challenging,  and
one might therefore expect that the more general problem of estimating the
value of $K$ will be hard as well. On the other hand, the previously unknown
connection between circle graphs and permutations might lead to new useful
insights.

\section{Open problems and further directions}

We have seen that the class $\Av(\sigma)$ is splittable when $\sigma$ is
decomposable (up to small exceptions), and unsplittable when $\sigma$
is a simple permutation. Obviously, the most natural open problem is to extend
this study to the remaining permutation patterns.

\begin{problem}
 For which pattern $\sigma$ is the class $\Av(\sigma)$ splittable?
\end{problem}

The number of simple permutations of order $n$ is asymptotically
$\frac{n!}{e^2}(1-\frac{4}{n}+O(n^{-2}))$, as shown by Albert, Atkinson and
Klazar~\cite[Theorem 5]{aak} (see also \cite[sequence A111111]{sloane}). 
The number of decomposable permutations of order $n$ is easily seen to
be $n!(\frac{2}{n}+O(n^{-2}))$ (see \cite[sequence A003319]{sloane}). Thus both
these classes form a significant proportion of all permutations, and one might
wonder what splittability behavior we should expect for $\Av(\sigma)$ when
$\sigma$ is a `typical' pattern of large size. We may phrase this formally
as follows.

\begin{problem}
 What is the asymptotic probability that $\Av(\sigma)$ is splittable, assuming
that $\sigma$ is chosen uniformly at random among the permutations of order~$n$?
\end{problem}

For certain permutation classes, e.g. for $\Av(1342)$, our results imply
splittability, but only provide splittings whose parts are defined by avoidance
of rather large patterns. We may hope that this is an artifact of our proof
technique, and that these classes admit more `natural' splittings. It is
possible to show that every permutation of order at most 3 is unavoidable for
$\Av(1342)$, as is every layered permutation or a complement of a layered
permutation. On the other hand, it is not clear whether e.g. the permutations
1423 or 2413 are unavoidable in $\Av(1342)$.

\begin{problem} Which permutations are unavoidable in the class $\Av(1342)$?
Is, e.g., $1423$ or $2413$ unavoidable in $\Av(1342)$? Does the class
$\Av(1342)$ even admit a splitting with all parts of the form $\Av(1423)$
or~$\Av(2413)$?
\end{problem}

We showed in Lemma~\ref{lem-ama} that an unsplittable class of
permutations is necessarily 1-amalgamable. This offers a simple approach to
prove splittability of a class $C$ by exhibiting an example of permutations
$\pi,\sigma\in C$ that fail to amalgamate in~$C$.  Although we never used
such approach in the present paper, we believe that amalgamability is a concept
worth exploring.

For some splittable classes, it is easy to see that they fail to be
1-amalgamable: e.g., a class of the form $\Av(\rho\oplus1\oplus\tau)$
has no amalgamation identifying the leftmost element of $\rho\oplus 1$ with the
leftmost element of $1\oplus\tau$, showing that the class 
$\Av(\rho\oplus1\oplus\tau)$ is not 1-amalgamable.
However, for more general splittable classes, there does not seem to be any
such obvious argument.

\begin{problem}
Is there an example of a splittable permutation class that is 1-amalgamable?
\end{problem}

We may also ask about higher-order amalgamations or Ramsey properties. As we
pointed out in the introduction, by results of Cameron~\cite{cam_per} and
B\"ottcher and Foniok~\cite{bofo}, any 3-amalgamable permutation class is 
amalgamable and Ramsey. It is not hard to verify that some classes, e.g.
the class $\Av(132)$, or the class $\Av(2413,3142)$ of separable
permutations, are 2-amalgamable but not 3-amalgamable. Beyond that, we do not
know much about 2-amalgamable classes.

\begin{problem}
Which permutation classes are 2-amalgamable? Are there infinitely many of them?
\end{problem}

\section{Acknowledgements}

We are grateful to Martin Klazar and Anders Claesson for valuable
discussions and helpful comments. 

\bibliographystyle{plain}

\end{document}